\documentclass[11pt,reqno]{article}
\usepackage{amsmath,amsthm,amssymb,cite}
\usepackage{mathrsfs}
\usepackage{graphicx}
\usepackage[title]{appendix}
 \newtheorem{theorem}{Theorem}[section]
    \newtheorem{proposition}[theorem]{Proposition}
    \newtheorem{lemma}[theorem]{Lemma}
    \newtheorem{corollary}[theorem]{Corollary}
   \numberwithin{equation}{section}
    \numberwithin{theorem}{section}
    \theoremstyle{definition}
\newtheorem{remark}[theorem]{Remark}{}
 \newtheorem{Definition}[theorem]{Definition}{}

    \newcommand{\nn}{\nonumber}

    \newcommand{\bb}[1]{{\mathbb #1}}
    
   \renewcommand{\epsilon}{\varepsilon}

    \newcommand{\e}{\mbox{e}}
    
\begin{document}
     \title {An expansion for the number of partitions of an integer}
    \author {Stella Brassesco  \and Arnaud Meyroneinc}
     \date{}
     \maketitle
     
    \begin{center} {Departamento de Matem\'aticas,
    Instituto Venezolano de Investigaciones Cient\'{\i}ficas, Apartado Postal
    20632 Caracas
    1020--A, Venezuela \\
  sbrasses@ivic.gob.ve\quad ameyrone@ivic.gob.ve}
   \end{center}
   
   \bigskip
   
   {\it Dedicated to the memory of Luis B\'aez Duarte}
  
  \bigskip

      \begin{abstract} 
    
    We  consider $p(n)$ the number of partitions of a natural number
      $n$, starting from an expression derived by L. B\'aez-Duarte in \cite{LBD}
      by relating its generating function  $f(t)$ with 
      the characteristic functions of a family of sums of independent 
      random variables indexed by $t$.   
       The asymptotic formula for   $p(n)$  follows then from a local central limit theorem as
        $t\uparrow 1$ suitably with $n\to \infty$.
       We take further that  analysis and compute formulae for  the terms that
       compose that expression, which accurately  approximate them as $t\uparrow 1$. 
         Those include the generating function $f$ and the
        cumulants of the random variables. After developing an asymptotic series expansion
         for the integral term we obtain 
      an expansion for $p(n)$ that can be simplified as follows: for each $N>0$, 
    \[
     p(n)=\frac{2\pi^2}{3\sqrt 3}
     \,\frac{\e^{r_n}}{(1+2\,r_n)^2}\,
     \big(1-\sum_{\ell= 1}^N\,\frac{D_\ell}{(1+2\,r_n)^\ell }+\,\mathcal R_{N+1}\,\big).
     \]
      The coefficients $D_\ell$ are positive and  have  simple   expressions as finite  sums of
     combinatorial numbers,  $r_n=\sqrt{\frac {2\pi^2}{3}\,(n-\frac{1}{24})+\frac 14}$ 
    and the remainder satisfies $n^{N/2}\,\mathcal R_{N+1} \to 0$ as $n\to \infty$.  

The cumulants are given by series of rational
      functions and the approximate formulae obtained  could be also  of independent interest in
       other contexts.   
       \end{abstract}
    \noindent {\footnotesize Keywords: Integer partitions, Asymptotic expansions, Asymptotic formulae, 
    Central limit theorem, Cumulants.\\
      11P82 (05A16 05A17)}
     \section {Introduction}
       For each  $n\in \bb N$,  consider  $p(n)$ the partition number of $n$, that is,
        the number of non increasing positive  integer sequences such that the sum of
        its terms adds to $n$. The study of partitions goes back to L. Euler,
         who established
        in particular the generating function $f$ of the sequence $p(n)$, 
           	\begin{equation}
      	\label{eulerf}
      	f(t)=\sum _{n\ge 0} p(n)\, t^n =\prod _{j\ge 1}\frac{1}{1-t^j},\quad  t\in [0,1).
      	\end{equation}
        
 In a celebrated article published a century ago \cite{HR18}, 
  G.H. Hardy and S. Ramanujan   considered the order of magnitude 
   of $p(n)$ as $n\to \infty$, and obtained an asymptotic series expansion 
   which in particular yields the 
  asymptotic formula 
  \begin{equation} 
  \label{hra}
  p(n)\sim \frac{\e^{\sqrt{\frac{2\,\pi^2}{3}\,n}}}{4\,n\,\sqrt 3}\quad \mbox{as } n\to \infty.
  \end{equation}
  H. Rademacher in \cite{Rad37} considered a modification of  the  asymptotic  series 
 that turns it convergent, and yields  a  value for 
  $p(n)$ that differs by decimals from its  exact value 
   when taking of the order of  $\sqrt n$ terms for large $n$. 
  
  Much more recently,  J.H. Bruinier and K. Ono in \cite{BO13} deduced
    an exact formula for 
  $p(n)$ that expresses it as a finite sum of algebraic quantities.
 In  \cite{DM} M. Dewar and
     R. Murty developed this formula to derive  an asymptotic expansion
       in the spirit of that of Rademacher.

L. B\'aez Duarte in \cite{LBD} considered a formula for $p(n)$, that, when approximating 
some of the terms  with the aid of  the Euler--Maclaurin
summation formula, yields \eqref{hra} as a consequence of a local central limit theorem.

  With that formula as starting point, we develop an asymptotic  expansion 
  for $p(n)$, consisting of a factor growing exponentially in $\sqrt{n}$,
 times  an asymptotic series expansion  given in terms of inverse powers of a certain quantity 
that is very nearly a rational function of $n$ growing as $\sqrt n$,  while the coefficients are 
explicitly given as sums of combinatorial numbers. 
Although the series is shown to be  convergent for any $n\in \mathbb N$, 
it is  really an asymptotic expansion, with errors 
of the order $o(n^{-N/2})$ as  $n\to \infty$ when taking $N$ terms.

 These results  are derived from 
an essentially real variable analysis, avoiding the 
difficulties due to 
the complicated structure of the function $f$ near
  the border of the unit circle in $\mathbb C$. The final expansion can be simplified 
  to obtain  an explicit expression in terms of rational functions of $n$.  
 
We compare
  the approximations  resulting  by truncating  the simplified expansion 
   with the true value of $p(n)$ for 
  several values of $n$, to illustrate the results.

     \section { Preliminaries and statement of results} 
 \label{SPS}      
As discussed in  \cite{LBD} (see also the references therein),
 the series in   \eqref{eulerf} determines for each  $t\in (0,1)$
 a random variable  $X(t)$ defined on some probability space and  taking  values on $\bb N$
 by taking
      $\bb P\big(X(t)=n\big)=\frac{p(n)\,t^n}{f(t)}$. We denote by $\mathbb P$ the  probability
       and by $\mathbb E$ the expectation
       with respect to $\mathbb P$. 
  The characteristic function of $X(t)$  is directly computed: 
      \begin{equation}
      \label{deffi}
      \varphi^{}_{X(t)}(\theta):=\bb E\, \e^{i\,\theta X(t)}=\frac{f(\e^{i\theta}\,t)}{f(t)}.
      \end{equation}
      Each factor $f_j(t):= \frac{1}{1-t^j}=\sum_{k\ge 0}t^{jk}$
       in \eqref{eulerf} determines analogously 
      a random variable $X_j(t)$ taking values  on the multiples of
       $j$:\[
       \bb P(X_j(t)=kj)=\frac{t^{jk}}{f_j(t)},
       \] whose characteristic function is 
      $\varphi^{}_{X_j(t)}(\theta)=\frac{f_j(\e^{i\theta}t)}{f_j(t)}$. 
      It follows then that the random variable $X(t)$ can be expressed
       as the infinite 
       sum of the  independent random variables $X_j(t)$.  
  Moreover, $Y_j(t)=\frac{X_j(t)}{j}$ is a random variable
   with geometric distribution with parameter $(1-t^j)$. In other words, 
   \begin{equation}
      \label{xsum}
      X(t)=\sum _{j\ge 1} j\, Y_j(t),
      \end{equation} 
      with $Y_j(t)$ independent   geometric random variables. 
    The mean and variance of $X(t)$ are directly computed,
      \begin{align}
      \label{mt}
    \kappa_1(t)&:=\bb E X(t)=  \sum _{j\ge 1} j\, \bb E Y_j(t)\,=
     \sum _{j\ge 1} j\,\frac{t^j}{1-t^j}
    \\
    \label{sigmat}
    \kappa_2(t)&:={\rm Var} (X(t))= \sum _{j\ge 1} j^2\,{\rm Var} (Y_j(t))\,=
    \sum _{j\ge 1} j^2\,\frac{t^j}{(1-t^j)^2}.
    \end{align}
    On the other hand  \eqref{deffi} may be rewritten in terms of the sequence $p(n)$ 
    \[
    \varphi^{}_{X(t)}(\theta)=\frac{1}{f(t)}\sum_{n\ge 1}\e^{i\theta n}\,p(n)\,t^n ,
    \]
    which yields by Fourier inversion 
    \begin{equation}
    \nn
    p(n)=\frac{f(t)}{2\pi t^n}\int^{\pi}_{-\pi}\e^{-in\theta}\varphi^{}_{X(t)}(\theta) d\theta. 
    \end{equation}
 Observe that the series  \eqref{mt} and \eqref{sigmat}
  are convergent for $t\in[0,1)$, and that 
     both $\kappa_1(t)\uparrow \infty$ and $\kappa_2(t)\uparrow \infty$ as $t\uparrow 1$. 
     We obtain indeed the precise asymptotic behaviour as $t\uparrow 1$  
   of $\kappa_1(t)$ and $\kappa_2(t)$ in Corollary \ref{cormsigma}. 
     Let us define  then the sequence
     \begin{equation}
     \label{deftn}
     t_n =\{ \mbox {the unique solution to $\kappa_1(t_n)=n$}\},
     \end{equation}
     and note  that $t_n \uparrow 1$ as $n\to \infty$. 
     Substitute  and  change variables in 
    the integral above to obtain 
    \begin{equation}
    \label{pn}
    p(n)=\frac{f(t_n)}{2\pi \sigma(t_n)\, t_n^n}
    \int^{\pi\sigma(t_n)}_{-\pi\sigma(t_n)}\bb E\, \e^{i\theta\,Z(t_n)}\, d\theta,
    \end{equation}
    where  $Z(t_n)$ is obtained by suitably centering and scaling $X(t_n)$:
     \begin{equation}
     \label{zt}
Z(t_n):=\frac{X(t_n)-\kappa_1(t_n)}{\sigma(t_n)}=\frac{X(t_n)-n}{\sigma(t_n)},
\end{equation}
and we have denoted as usual by  $\sigma(t)=\sqrt {\kappa_2(t)}$ the standard deviation of $X(t)$.  
From \eqref{xsum},   \eqref{mt},  \eqref{sigmat} and \eqref{deftn},
  the  central limit theorem implies that  
 $Z(t_n)$ is asymptotically normal (as $n\to \infty$). 
Moreover, in \cite{LBD}, it is proved that a local central limit theorem holds
and the integral above converges to $\sqrt {2 \pi}$,  what yields after
 determining the asymptotics of the terms 
in \eqref{pn} as $n\to \infty$ the asymptotic formula \eqref{hra}. 

We obtain accurate approximations (as $n\to \infty$) for the terms in the factor, and 
an asymptotic series   expansion for the integral term in \eqref{pn}, following by expanding the integrand 
in terms of the cumulants of the random variables. 
The resulting quantities  do not seem to  have   simple expressions in terms of $n$,
but we show that they are well approximated by elementary functions. Indeed, 
let us define 
\begin{equation}
\label{defrn}
r_n:=\sqrt{\frac {2\pi^2}{3}\,(n-\frac{1}{24})+\frac 14}.
\end{equation}
 The expansion that we obtain turns out to be given in terms of  $c_n^2:=\kappa_2(t_n)\,|\log t_n|^2$, 
 which is  
 very well approximated (as $n\to\infty$) by $r_n$. Precisely, it follows from  Lemma \ref{lemcn} that
  for some positive constant $C$,  
\begin{equation}
\label{approxcn}
|c_n^2-
 r_n|
  \le Cn\,\e^{-2\pi\sqrt{24n-1}}. 
\end{equation}
Most of our formulae are given in terms of  $|\log t_n|$, that is shown 
in Corollary \ref{cortn} to satisfy:   
  \begin{equation}
 \label{approxlogtn}
 \big|\,|\log t_n|\,-
 \,\frac{2\pi^2}{3\,\big(1+2\,r_n\,\big)}\,\big|
 \le C\,\e^{-2\,\pi\sqrt{24n-1}}.   
\end{equation}
 In terms of those quantities, we can now state our main results. 
\begin{theorem}
\label{Thpp}
For any given $N>0$, $p(n)$ satisfies 
\begin{equation}
\label{asymexp}
p(n)=
 \,\frac{ {\rm e}^{r_n}\,|\log t_n|^{\frac 32}}{\sqrt {2}\,\pi\,\sqrt{1+2\,c_n^2}}\,\big(1-
\sum_{\ell= 1}^N\,\frac{D_\ell\,}{(1+2\,c_n^2)^\ell }\,
+\widetilde {\mathcal R}_{N+1}\,\big),
\end{equation}
 where: 
 \begin{itemize}
  \item[$\circ$] The coefficients $D_\ell$ are given by 
  \begin{equation}
  \label{defDl}
  D_\ell=(-1)^{\ell+1}\,\frac{(\ell+1)}{4^\ell} \sum_{k=0}^{\ell+1}(-1)^k\, 2^k\,\binom{2\ell}{k}
\frac{1}{(\ell+1-k)!},
\end{equation}
and satisfy $0<D_\ell<2\,3^{\ell-1}$. 
\item[$\circ$] The error $\widetilde {\mathcal R}_{N+1}$ satisfies
 $n^{N/2}\,\widetilde {\mathcal R}_{N+1}\to 0$ 
as $n\to \infty$.
   \end{itemize}
\end{theorem}

It is possible to clear up the expression \eqref{asymexp} by  
 replacing $1+2c_n^2$ and  $|\log t_n|$ by the  approximations that follow from 
 \eqref{approxcn} and \eqref{approxlogtn}.
  After verifying that the errors introduced by doing that are sufficiently 
 small to be included in the error terms of the expansion, 
 the following alternative expression  is obtained. 
 
 \begin{proposition} 
 \label{propalt}
 For any given $N>0$, $p(n)$ satisfies 
\begin{equation}
\label{asymexpa}
p(n)=
\frac{2\pi^2}{3\sqrt 3}\,
     \,\frac{{\rm e}^{r_n}}{(1+2\, r_n)^2}\,
     \big(\,1-\sum_{\ell= 1}^N\,\frac{D_\ell}{(1+2\,r_n)^\ell }
+{\mathcal R}_{N+1}\,\big),
\end{equation}
 where: 
 \begin{itemize}
 \item[$\circ$] The coefficients $D_\ell$ are those in \eqref{defDl}.
 \item[$\circ$] The error $\mathcal {R}_{N+1}$ satisfies
  $n^{N/2}\,\mathcal { R}_{N+1}\to 0$ as $n\to \infty$.
   \end{itemize}
\end{proposition}
The expansion \eqref{asymexpa} is indeed less precise than that in Theorem \ref{Thpp}, but
its terms are very simple to implement. We report in Table \ref{table}  a comparison  of the 
 approximate  values for  $p(n)$ obtained by computing the sum in   \eqref{asymexpa} 
 with  $N=17$ for several values of $n$. The results appear reasonably accurate
 even for the  small values of $n$ that were considered. 
 
 Both statements are proved in Section \ref{finale}, after 
 developing asymptotic formulae for each of the terms in \eqref{pn}. 
 
In Section  \ref{aexps} we consider the terms in the factor. 
The behaviour of 
$f(t)$ as $t\uparrow 1$ is derived from a functional formula, which  by means of a
 recurrence permits also to 
deduce simple  formulae for the first two cumulants that approximate
 them very accurately as $t\uparrow 1$.   Those in turn render 
 very good approximations for the sequence  $t_n$. 
 That analysis turns out the exponential term in \eqref{asymexp}. 
 
 In Section \ref{tit} we obtain an asymptotic series  expansion for the integral term in \eqref{pn}, 
 which produces the  sum in the expansion of $p(n)$. The proof is based on
 an expansion of the integrand
 in terms of the cumulants of the random variables.
 In the Appendix, we obtain functional formulae for all those cumulants, and
  deduce from them  their precise asymptotics as $t\uparrow 1$, 
  which in turn permits to control the errors in the expansion. We derive  
 as well  some properties  of the cumulant generating function needed in the analysis.
Explicit formulae for all the cumulants of $X(t)$, $t\in [0,1)$ are also obtained, 
 similar to those 
 for the first two cumulants in  \eqref{mt} and \eqref{sigmat} and  whose numerators are given in terms of Eulerian polynomials.

\section {Asymptotic formulae for the terms in the factor}
\label{aexps}
In this section, we obtain  expressions for  the terms in the factor in \eqref{pn}, 
suitable for determining their  behaviour for large $n$, which corresponds to   
$t_n$ close to $1$. The asymptotics of $f(t)$ as $t\uparrow 1$ is derived  from the functional formula
for $\log f(t)$ stated in the next result.   
\begin{lemma}
 \label{alogf}
 For $t\in (0,1)$, 
 \begin{align}
 \nn
\log f(t)\,=&\,\frac{\pi^2}{6|\log t|}-\log(\sqrt{2\pi})-\frac{1}{24}|\log t|+
\frac 12 \log\,|\log t|
\\
\label{logft1}
&+\log f({\rm e}^{-\frac{4\pi^2}{|\log t|}}).
\end{align}
 \end{lemma}
\begin{proof} The following identity for $\alpha, \beta >0$, $\alpha \beta=\pi^2$
 is Corollary (ii) in \cite[p.256]{RNII}:
 \[\e^{(\alpha-\beta)/12}=\big(\frac{\alpha}{\beta}\big)^{\frac 14}
 \prod_{j=1}^\infty\frac{1-\e^{-2\alpha\,j}}{1-\e^{-2\beta\,j}}.
 \] 
The substitution $\e^{-2\beta}=t$ and correspondingly  $\alpha =2\pi^2/|\log t|$ in the  above  product  yields from \eqref{eulerf}  
 an expression  for $f(t)/f({\rm e}^{-\frac{4\pi^2}{|\log t|}})$, that after taking logarithm
and reordering is exactly 
\eqref{alogf}.
\end{proof}

Let us observe now that, as $t\uparrow 1$, the last term in \eqref{alogf} goes to zero.
 Moreover,  starting from \eqref{eulerf} we can estimate for $\lambda \in (0,1)$,
 \begin{equation}
 \nonumber
 \partial_\lambda\big(\log f(\lambda)\big) = \sum_{j=1}^\infty j\,
 \frac{\lambda^{j-1}}{1-\lambda^j}\le \frac{1}{1-\lambda}
 \sum_{j=1}^\infty j\,
 \lambda^{j-1}=\frac{1}{(1-\lambda)^3}
 \end{equation}
 to obtain, as an application of  the mean value theorem that for some $ \xi\in (0,\lambda)$
 \begin{equation}
 \nonumber
 \log f(\lambda)=\lambda\, \partial_\lambda
 \big(\log f(\lambda)\big\vert_{\lambda=\xi}\big)\le \frac{\lambda}{(1-\lambda)^3}. 
 \end{equation}
 In particular, if  we take any $b>1$ and  $t\in (\e^{-\frac{4\pi^2}{\log b}},1)$ then 
  $\e^{-\frac{4\pi^2}{|\log t|}}<\frac 1b$, and we conclude that 
 \begin{equation}
\label{eE0}
0<E_0(t):=\log f(\e^{-\frac{4\pi^2}{|\log t|}})
\le (\frac {1}{1-\frac 1b})^3\, \e^{-\frac{4\pi^2}{|\log t|}}.
\end{equation}
For instance, taking $b=2$ we obtain that, for any $t>\e^{-4\pi^2}\approx 10^{-17}$, 
\begin{equation}
\label{E0le}
0<E_0(t)< 8\,\e^{-\frac{4\pi^2}{|\log t|}}.
\end{equation}
We may thus regard \eqref{alogf} as an asymptotic formula for $\log f(t)$ 
as $t\uparrow 1$, which  we  write in the form 
\begin{equation}
\label{lfa} 
\log f(t)=\frac{\pi^2}{6|\log t|}-\log(\sqrt{2\pi})-\frac{1}{24}|\log t|+
\frac 12 \log\,|\log t|+E_0(t),
\end{equation} 
with $E_0(t)$ satisfying \eqref{eE0}. We also have from  the above estimations  that 
 $E_0(t)\asymp\e^{-\frac{4\pi^2}{|\log t|}}$. The symbol $\asymp$
  denotes asymptotic logarithmic order, as formalized in the following definition.
  We find this  kind of asymptotic equivalence  along the way in the sequel. 

\begin{Definition}
\label{defasym}
We say that a function $E(t)$ defined for $t$ in the interval $(0,1)$
is logarithmically equivalent to ${\rm e}^{-\frac{A}{|\log t|}}$ for
 some $A>0$ as $t\uparrow 1$
 and denote 
\[
E(t)\asymp {\rm e}^{-\frac{A}{|\log t|}}
\]
if
\[
\underset{t\uparrow 1}{\lim}|\log t|\big( \log |E(t)|\big)=-A. 
\] 
\end{Definition}

 We proceed to 
deduce formulae for   the mean $\kappa_1(t)$ and variance
 $\kappa_2(t)$ suitable  for the analysis  as $t\uparrow 1$, 
 and used in Corollary \ref{cortn}  to derive  precise  estimates for $t_n$
 as $n\to \infty$ from \eqref{deftn}. 
 We recall that  $\kappa_1(t)$ and $\kappa_2(t)$ are given by the series
 \eqref{mt} and \eqref{sigmat}, as already   computed in the previous section,
\begin{equation}
\label{cum12}
    \kappa_1(t)= \sum _{j\ge 1} j\,\frac{t^j}{1-t^j},\qquad \qquad 
    \kappa_2(t)\,=
    \sum _{j\ge 1} j^2\,\frac{t^j}{(1-t^j)^2}.
\end{equation}
\begin{lemma}
\label{cumulants}
The mean $ \kappa_1(t)$ and variance $\kappa_2(t)$ satisfy the following functional equations,
\begin{align}
\label{m(t)}
 \kappa_1(t)=&\,\frac{\pi ^2}{6|\log t|^2}\,-\,\frac {1}{2\,|\log t|}\,
+\frac{1}{24}-\frac{4\pi^2}{|\log t|^2}\, \kappa_1({\rm e}^{-\frac{4\pi^2}{|\log t|}}),
\\
\label{sigma2}
 \kappa_2(t)=&\,\frac{\pi^2}{3|\log t|^3}-\frac{1}{2|\log t|^2}+
\big(\frac{4\pi^2}{|\log t|^2}\big)^2
 \kappa_2({\rm e}^{-\frac{4\pi^2}{|\log t|}})
\\ \nn
&-\frac{8\pi^2}{|\log t|^3}\, \kappa_1({\rm e}^{-\frac{4\pi^2}{|\log t|}}).&
\end{align}
\end{lemma}
\begin{remark}
The first series in  \eqref{cum12} can be computed explicitly
   in the case $t=\e^{-2\pi}$, as follows from  formula (8.3), page 255 of B.C. Berndt's edition of 
   S. Ramanujan's notebooks  \cite{RNII}:
\begin{equation}
\nn
\sum_{j\ge1} j\,\frac {\e^{-2\pi j}}{(1-\e^{-2\pi j})}=\frac{1}{24}-\frac{1}{8\pi}. 
\end{equation}
Indeed, the identity \eqref{m(t)} follows  from that formula  as well. 
An account of series like those in \eqref{cum12}, that appear also in other contexts, 
can be found in the references therein. Some examples are mentioned in \cite{eg}. 
A closed form  of $\kappa_1(t)$   for other values of $t$ does not appear to be known.
\end{remark}

The formulae in the previous lemma yield precise asymptotics for 
$\kappa_1(t)$ and $\kappa_2(t)$  as $t\uparrow 1$, as  we state next. 
\begin{corollary}
\label{cormsigma}
\begin{align}
\label{mta}
\kappa_1(t)=&\,\frac{\pi ^2}{6|\log t|^2}\,-\,\frac {1}{2\,|\log t|}\,
\,+\frac{1}{24}+E_1(t),
\\
\label{sigma2a}
\kappa_2(t)= &\,\frac{\pi^2}{3|\log t|^3}-\frac{1}{2|\log t|^2}\,+E_2(t),
\end{align}
where
\[
E_j(t)\asymp {\rm e}^{-\frac{4\pi^2}{|\log t|}}\quad j=1,2 \qquad \mbox{as } t\uparrow 1.
\]
\end{corollary} 
Before proving  Lemma \ref{cumulants} and the above corollary, 
let us recall  the notion of cumulants of a random variable and some of their properties.  
  
Given a random variable $X$ with characteristic function 
$\varphi^{}_X(\theta)=\bb E\,\e^{i\theta\,X} $, 
its cumulants (or semi-invariants) are the coefficients $\kappa_j$
in the Taylor  series expansion 
of $K_X(\theta):=\log \varphi^{}_X(\theta)$ (provided such expansion exists),
\begin{equation}
\label{gcum}
K_X(\theta)=\sum_{j\ge1}\kappa_j\frac{(i\theta)^j}{j!}.
\end{equation}
We consider here and in the sequel the principal value of the logarithm.  
Therefore, 
$
i^j\kappa_j=\partial_\theta^{(j)}\,K_X(\theta)\vert _{\theta=0}
$ 
and $K_X$ is referred to as  the cumulant generating function of the random variable $X$.
The first cumulant $\kappa_1$ results to be  the mean, and the second one $\kappa_2$  the
variance, when those exist. We refer to \cite{Sh} for further properties of $K$.

In the case of the random variables $X(t)$, $t\in (0,1)$ considered in the introduction,
from \eqref{deffi}
\[ 
K_{X(t)}(\theta)=\log \frac{f(\e^{i\theta}t)}{f(t)}, 
\] 
and it is not difficult to verify from  this last expression that 
its cumulants, that we denote by $\kappa_j(t)$, are finite 
 and  can be computed from  the simple recursion  formulae 
\begin{equation}
\label{rcum1j}
\kappa_1(t)=t\partial_t \log f(t), \quad \kappa_{j+1}(t)=t\partial_t \kappa_{j}(t), \quad j\ge 1.
\end{equation}
Starting from 
\begin{equation}
\nn 
\log f(t)= \sum _{\ell\ge 1}\log \frac{1}{1-t^\ell}=\sum _{\ell\ge 1}\frac{t^\ell}{\ell\,(1-t^\ell)},
\end{equation}
we recover from the first identity  the  series in \eqref{cum12}
 for $\kappa_1(t)$ and $\kappa_2(t)$.  Another couple of forms for those series 
 follows from the second identity; we recall them  here for further reference:
 \begin{equation}
 \label{cum12a}
 \kappa_1(t)= \sum _{\ell\ge 1}\frac{t^\ell}{(1-t^\ell)^2}, \qquad \qquad 
    \kappa_2(t)\,=
    \sum _{\ell\ge 1} \ell\,\frac{(t^\ell+t^{2\ell})}{(1-t^\ell)^3}.
  \end{equation} 
 It is not difficult to conclude from \eqref{cum12} and  the recurrence 
\eqref{rcum1j} that,  besides being finite for each $ t\in [0,1)$,  the cumulants $\kappa_j(t)$, $j\ge 1$ 
satisfy  $\kappa_j(t)\uparrow \infty$ as $t\uparrow 1$. Explicit formulae for $j>2$  are computed in the Appendix. 

\begin{proof}[Proof of Lemma \ref{cumulants}]
Observe that  for any real function $H$ with derivative $H'$, 
if $t\in(0,1)$, 
 \begin{equation}
 \label{fH}
 t\,\partial_t\, H(|\log t|)=-H'(|\log t|),
 \end{equation} 
 whenever both sides are defined. 
 Apply next $t\,\partial_t$ to   \eqref{logft1} to obtain formula  \eqref{m(t)} directly from 
 \eqref{rcum1j}, and again 
 $t\,\partial_t$ to  \eqref{m(t)} to obtain \eqref{sigma2}.
\end{proof}

\begin{proof}[Proof of Corollary \ref{cormsigma}]
Note that from Lemma \ref{cumulants}, \eqref{mta}  and \eqref{sigma2a},  
\begin{align}
\label{defE1}
E_1(t)=&-\frac{4\pi^2}{|\log t|^2}\,\kappa_1(\e^{-\frac{4\pi^2}{|\log t|}})
\\
\label{defE2}
E_2(t)=&\,\big(\frac{4\pi^2}{|\log t|^2}\big)^2
 \kappa_2(\e^{-\frac{4\pi^2}{|\log t|}})
-\frac{8\pi^2}{|\log t|^3}\, \kappa_1(\e^{-\frac{4\pi^2}{|\log t|}}).
\end{align}
It is thus  enough to show that, given $T\in (0,1)$,   there are positive
  constants $c_j$ and $C_j$ such that 
if $t\in (T,1)$,
\begin{equation}
\label{cjcj}
c_j\,\e^{-\frac{4\pi^2}{|\log t|}}\le\kappa_j(\e^{-\frac{4\pi^2}{|\log t|}})
\le\, C_j\, \e^{-\frac{4\pi^2}{|\log t|}}\quad
j=1,2.
\end{equation}

The lower bounds follow by estimating the series  \eqref{cum12} by 
their first term, as all the  terms are positive, and the denominator
 is bounded away from zero if $t\in (T,1)$. 
 The upper bound 
follows from the mean value theorem as in the proof of \eqref{eE0}. 
Indeed, for $j=1$,  observe that, from \eqref{cum12a}, 
\[
\partial_{\lambda}\kappa_1(\lambda)=
\sum _{\ell=1}^\infty \ell\,\frac{\lambda^{\ell-1}\,(1+\lambda^\ell)}{(1-\lambda^\ell)^3}\le
\frac{2}{(1-\lambda)^3}\,\partial_{\lambda}
\big(\frac{\lambda}{1-\lambda}\big)=\frac{2}{(1-\lambda)^5},
\]  
and then, for any  $b>1$ and
  $t\in (\e^{-\frac{4\pi^2}{\log b}},1)$, 
\begin{equation}
\nn
\kappa_1(\e^{-\frac{4\pi^2}{|\log t|}})\le
 \frac{2}{(1-\frac 1b)^5}\,
\e^{-\frac{4\pi^2}{|\log t|}}.
\end{equation} 
In particular, if we choose $b=2$, we can see that for any $t>\e^{-4\pi^2}\approx10^{-17}$,
\begin{equation} 
\label{k1num}
\kappa_1(\e^{-\frac{4\pi^2}{|\log t|}})\le 64\, \e^{-\frac{4\pi^2}{|\log t|}}.
\end{equation}
The same procedure yields the upper bound in \eqref{cjcj} for $j=2$. 
 \end{proof}
 
\smallskip

We deduce next good approximations (as $n\to \infty$) to convenient multiples of  $|\log t_n|$ and its reciprocal, in terms of the rational function 
$r_n=\sqrt{\frac {2\pi^2}{3}\,(n-\frac{1}{24})+\frac 14}$ already defined in \eqref{defrn}.   
  Recall that $t_n$ was defined in
\eqref{deftn} as the solution to $\kappa_1(t_n)=n$, and it is therefore an increasing
sequence converging to 1 as $n\to \infty$. 

\smallskip

\begin{corollary}
\label{cortn}
The parameter $t_n$ introduced  in \eqref{deftn} satisfies 
\begin{equation}
\label{tnle}
{\rm e}^{-\pi/\sqrt{6\,(n-\frac {1}{24})}}<t_n<1.
\end{equation}
Moreover, there is a positive constant $C$ such that  
\begin{align}
\label{tna}
\big|\,\frac{\pi^2}{3|\log t_n|}\,-\,
(\frac 12\,+r_n)\,\big|&\le 
 C\sqrt n\,{\rm e}^{-2\,\pi\sqrt{24n-1}},
 \\ \label{ltnrn}
  \big|\,|\log t_n|\,-
 \,\frac{2\pi^2}{3\,\big(1+2\,r_n\,\big)}\,\big|
 &\le C\,{\rm e}^{-2\,\pi\sqrt{24n-1}}.  
\end{align}
\end{corollary}

\begin{proof} 
From \eqref{deftn} and  \eqref{m(t)} we know
\begin{equation}
\label{mt2}
\frac{\pi ^2}{6|\log t_n|^2}\,-\,\frac {1}{2\, |\log t_n|}\,+\,\frac{1}{24}-
\frac{4\pi^2}{|\log t_n|^2}\, \kappa_1(\e^{-\frac{4\pi^2}{|\log t_n|}})=n.
\end{equation}
Picking only the positive terms on the left hand  side we obtain 
\begin{equation}
\label{tngn}
\frac{\pi ^2}{6|\log t_n|^2}\,+\,\frac{1}{24}> n,
\end{equation}
which yields the lower bound   in \eqref{tnle}, while the upper bound is  
a direct consequence of   the definition \eqref{deftn}.  
Since $t_n$ increases with $n$, from \eqref{tnle}
  \begin{equation}
  \label{tnge}
t_n>t_1>\e^{-\pi/2},\quad \forall n \ge 1
\end{equation}
thus the inequality \eqref{k1num} holds for any $n\ge 1$ when taking $t=t_n$. 
Furthermore, from \eqref{tngn}
\begin{equation}
\label{e4pi}
\e^{-\frac{4\pi^2}{|\log t_n|}}\le \e^{-2\pi \sqrt{24 n -1}},
\end{equation}
 and then from \eqref{k1num},  for each $n\in \mathbb N$ 
\begin{equation}
\label{hn}
h_n:=\kappa_1(\e^{-\frac{4\pi^2}{|\log t_n|}})\le 64\, \e^{-\frac{4\pi^2}{|\log t_n|}}
\le  64\, \e^{-2\pi\sqrt{24n-1}}.
\end{equation}
As $h_n$  decreases with $n$, we also  have that 
\begin{equation}
\label{hn1}
h_n\le 64\,\e^{-8\pi}< 2^{-30}\quad \forall n \ge 1.
\end{equation}
Denote $x:=\frac{  \pi^2}{3|\log t_n|}$ and  
multiply \eqref{mt2} by $\frac 23\,\pi^2$ to obtain that   $x$ satisfies 
\[
x^2(1-24\,h_n)-x-\frac 23\pi^2\big(n-\frac{1}{24}\big)=0.
\]
Observe that $h_n=\kappa_1(\e^{-12\,x})$ so it indeed   depends on $x$,
 but we know from \eqref{hn} that it is positive
and exponentially small. Solving the equation  we obtain  
\begin{equation}
\label{solvx}
 \frac{  \pi^2}{3|\log t_n|}=x=
 \frac{\frac 12\,+\sqrt{\frac{2}{3}\pi^2(n-\frac{1}{24})(1-24\,h_n)+\frac 14}}{(1-24\,h_n)}.
\end{equation}
In terms of $r_n$, let us write
\begin{equation}
\label{deltanrn}
\sqrt{\frac{2}{3}\pi^2(n-\frac{1}{24})(1-24\,h_n)+\frac 14}=r_n\,\sqrt{1-\delta_n}
\end{equation}
with
\begin{equation}
\label{defdeltan}
\delta_n=24\,h_n\,\big(1-\frac{9}{\pi^2(24\,n-1)+9}\big)\le 24\,h_n. 
\end{equation}
To estimate the difference in \eqref{tna}, from \eqref{solvx} we  write
\begin{multline}
 \nn
\frac{\pi^2}{3\,|\log t_n|} -(\frac12 +r_n)=\frac{\frac 12 +r_n\, \sqrt{1-\delta_n}}{1-24\,h_n} 
-(\frac12 +r_n) 
\\
=\frac{r_n\,(\sqrt{1-\delta_n}-1)+24\,h_n(\frac 12 +r_n)}{1-24\,h_n}
\end{multline}
whence, from \eqref{defdeltan}
\begin{equation}
\nn
\big|\,\frac{\pi^2}{3\,|\log t_n|} -
\big(\frac12 +r_n\,\big)\big|\le
\frac{50\, h_n\,r_n}{1-24\,h_n},  
\end{equation}
and the proof of \eqref{tna} follows  from \eqref{hn} and the definition of $r_n$.
The proof of \eqref{ltnrn} follows as well from \eqref{solvx} and the above estimates.  
\end{proof}
As consequence of the previous analysis, we obtain
an expression for the factor in \eqref{pn}, that
yields precise asymptotics as $n\to \infty$. 
 
\begin{corollary}
\label{Corfactor}
The factor in front   of the integral term in \eqref{pn} satisfies 
 \begin{equation}
\nn
\frac{f(t_n)}{2\pi \sigma(t_n) t_n^n}=\frac{1}{(2\pi)^{3/2}}
\frac{|\log t_n|^{1/2}}{ \sigma(t_n)} \,
{\rm e}^{\sqrt{\frac {2\pi^2}{3}\,(n-\frac{1}{24})+\frac 14}\,+
 \beta_n},
\end{equation}
where $\beta_n$ satisfies that, for some positive  constant $C$
\begin{equation}
\label{betale}
|\beta_n| \le  C \,{\rm e}^{-2\,\pi\sqrt{24n-1}}.   
\end{equation}
\end{corollary} 

\begin{proof} Let us write the factor
 in  \eqref{pn} as  
\begin{equation}
\nn
\frac{f(t_n)}{2\pi \sigma(t_n) t_n^n}=
\frac{1}{2\pi \sigma(t_n)}\,\e^{\log f(t_n)+n |\log t_n|}.
\end{equation}
From   \eqref{lfa} evaluated at $t=t_n$, 
\eqref{mt2} and the definition of $h_n$ \eqref{hn}
it follows  after collecting terms that 
\begin{align} 
\nn
\log f(t_n)+n |\log t_n|=&\frac{\pi^2}{3|\log t_n|}- \frac 12 +
\frac 12 \log\,|\log t_n|
-\log(\sqrt{2\pi})
\\ 
\nn
 + E_0(t_n)-\frac{4\pi^2}{|\log t_n|}&\,h_n
\\  \nn
=&\sqrt{\frac 23 \pi^2(n-\frac{1}{24})(1-24\,h_n)+\frac 14}
\\
\label{factor} 
+\frac 12 \log\,|\log t_n|
-&\log(\sqrt{2\pi}) +E_0(t_n)+\frac{4\pi^2}{|\log t_n|}\,h_n.
\end{align} 
The second identity follows from \eqref{solvx}. 
Next, observe that if we write the square root as in \eqref{deltanrn}, 
 after expanding 
 $\sqrt{1-\delta_n}=1-\frac{\delta_n}{2}+O(\delta_n^2) $ we obtain  
\begin{equation}
\nn
\sqrt{\frac 23 \pi^2(n-\frac{1}{24})(1-24\,h_n)+\frac 14}=
r_n\,\sqrt{1-\delta_n}
\\
=r_n\,
\big(1-12\, h_n\big)+o(h_n).
\end{equation}
Let us express, with the aid of Corollary \ref{cortn}
 \[
 \frac{4\pi^2}{|\log t_n|}\,h_n= 12\, h_n \frac{\pi^2}{3|\log t_n|}=
 12\, h_n \big(\frac 12 +\sqrt{\frac 23 \pi^2(n-\frac{1}{24})+\frac 14\,}\,\big) +O(h_n^2 \sqrt n).
 \]
Substitution of the last two formulae into 
\eqref{factor} yields 
\[
\log f(t_n)+n |\log t_n|= 
\sqrt{\frac {2\,\pi^2}{3}(n-\frac{1}{24})+\frac{1}{4}}\,+
 \frac 12 \log\,|\log t_n|
-\log(\sqrt{2\pi})\,+\beta_n, 
\]
with  $\beta_n=E_0(t_n)+o(h_n)+6\, h_n +O(h_n^2 \sqrt n) $.  
  
But we know  from \eqref{E0le} and \eqref{e4pi} that 
\[ 
0<E_0(t_n)< 8\,\e^{-2\pi\sqrt{24n-1}},
\]
which, together with  \eqref{hn}  yields \eqref{betale}. 
\end{proof}

\begin{remark} Good approximations to $\frac{|\log t_n|^{\frac 12}}{\sigma(t_n)}$   in terms of $n$  
follow  from Corollary \ref{cormsigma} and \eqref{tna}. In particular, the blunt estimate 
\begin{equation}
\label{seterm}
\frac{|\log t_n|^{\frac 12}}{\sigma(t_n)}\sim\frac{\pi}{2\sqrt 3\, n}
\end{equation}
is obtained.  We prefer to keep that term untouched here, as some simplifications 
occur when considering the integral term in the next section, 
and more precise estimates are presented there. 
\end{remark}

\section {The expansion of the integral term}
\label{tit}
Let us turn now to consider the integral in \eqref{pn}.
In \cite{LBD}, it is proved that 
\begin{equation}
\label{i2}
\int^{\pi\sigma(t_n)}_{-\pi\sigma(t_n)}
\bb E\, \e^{i\theta\,Z(t_n)}\, d\theta \to \sqrt{2\pi}\quad
\mbox{ as } n \to  \infty, 
\end{equation}
which, 
together with   Corollary  \ref{Corfactor} and \eqref{seterm} yields  the  asymptotic
formula \eqref{hra} from \eqref{pn}.  

We follow the approach of \cite{BM} to obtain an asymptotic series  expansion for
the above integral. The analysis here is more
involved, among other reasons due to the fact that  the terms are 
defined implicitly,  and we need to consider approximations. We develop the integrand
 in terms of the cumulants of the random variables $X(t)$.   
In the appendix, we obtain explicit expressions for each of the cumulants as series of rational functions
of $t$, whose numerators are given in terms of Eulerian polynomials. Simple functional  formulae
similar to those already obtained in the previous section  for the first two cumulants
 $\kappa_1(t)$ and $\kappa_2(t)$ 
are also  deduced from the recurrence relation \eqref{rcum1j}. They  yield accurate   approximations
 (as $t\uparrow 1$), which  permit to control the errors. The expansion results  
to be given  in inverse powers of   $1+2\,c_n^2$,  with $c_n^2$ defined  in Section \ref{SPS}:
 \begin{equation}
\label{defcn}
c_n^2=\big(\sigma(t_n)\, |\log t_n|\big)^2.
\end{equation}
It is easy to see from  Corollary
\ref{cormsigma} and \eqref{tna} that $c_n^2=O(\sqrt n)$.
It is indeed very nearly $r_n$, which was defined in \eqref{defrn}, 
 as precisely established in the following lemma. 

\begin{lemma}
\label{lemcn}
 The sequence  $c_n^2$
  satisfies that, for some positive constant $C$, 
\begin{equation}
\label{estdn}
-\,C\,\sqrt n\,{\rm e}^{-2\,\pi\sqrt{24n-1}}
 \le  c_n^2\,-\,r_n
 \le\,C\, n\,{\rm e}^{-2\,\pi\sqrt{24n-1}}.
\end{equation}
\end{lemma}
\begin{proof}
From \eqref{sigma2a},  \eqref{solvx}, \eqref{defE2} and \eqref{hn},  
\begin{multline}
\nn
c_n^2 =\frac{\pi^2}{3|\log t_n|}-\frac 12 + |\log t_n|^2\,E_2(t_n)= 
\sqrt{\big(\frac {2\pi^2}{3}(n-\frac{1}{24})+\frac{1}{4}\big)\,(1-\delta_n)}\\ +
\Big(\frac{4\pi^2}{|\log t_n|}\Big)^2\,\kappa_2(\e^{-4\pi^2/|\log t_n|}),
\end{multline}
with $\delta_n$ as defined in \eqref{defdeltan}. 
From \eqref{tna},  \eqref{cjcj} and  \eqref{e4pi}
we conclude that \eqref{estdn} holds. 
\end{proof}
\begin{theorem}
\label{Tint}
Given $N>0$, the integral in \eqref{i2} can be expanded as 
\begin{align}
\label{Thint}
\int^{\pi\sigma(t_n)}_{-\pi\sigma(t_n)}\!\bb E\, {\rm e}^{i\theta\,Z(t_n)}\, d\theta=
\sqrt{2\pi}\,\sqrt{\frac{2\,c_n^2}{1+2\,c_n^2}}
\Big(1-\sum_{\ell= 1}^N\,\frac{D_\ell}{(1+2\,c_n^2)^\ell }\Big)
+\mathcal E_{N+1},
\end{align}
where:
\begin{itemize}
\item[$\circ$]
The coefficients $D_\ell$ are
\begin{equation}
\label{defdl}
D_\ell=(-1)^{\ell+1}\,\frac{(l+1)}{4^\ell} \sum_{k=0}^{\ell+1}(-1)^k\, 2^k\,\binom{2\ell}{k}
\frac{1}{(\ell+1-k)!},  
\end{equation}
and they 
satisfy 
\begin{equation}
\label{Dele}
0<D_\ell<2\,3^{\ell-1}.
\end{equation}
\item[$\circ$]
 The error term $\mathcal E_{N+1}=\mathcal E_{N+1}(n)$
   satisfies  
\begin{equation}
\label{Eint}
n^{N/2}\,\mathcal E_{N+1} \to 0 \mbox{ as } n\to \infty.
\end{equation}
\end{itemize}
\end{theorem}

\begin{proof}
To obtain the asymptotics, it is convenient to  split  the integration interval 
$
\mathcal A:=\{\theta:\frac{|\theta|}{\sigma(t_n)}<\pi\}
$ into disjoint regions  $\mathcal A_{1}$ and $\mathcal A_{2}$,  
according to $\frac{|\theta|}{\sigma(t_n)}\le a\,(1-t_n)$ or not 
 for some positive    constant  $a<1$: 
 \[
\mathcal A =\mathcal A_{1}\cup \mathcal A_{2},\quad \mbox{ for } 
\mathcal A_{1}=\{\theta: \frac{|\theta|}{\sigma(t_n)}\le a\,(1-t_n)\}.
\]
The value of  $a$ is fixed hereafter and it is  not included  in the notation.
 
 Let us consider first the integral in $\mathcal A_{2}$, starting with 
 \[ 
 \big| \int\limits_{\mathcal A_2 }
  \bb E\, \e^{i\theta\,Z(t_n)}\, d\theta\,\big|\le
   \int\limits_{\mathcal A_2 }\!\big|\varphi^{}_{X(t_n)}
  (\theta/\sigma)\big| \, d\theta.
\]
We follow the computations in the proof of the  theorem in 
\cite[p.119]{LBD} to estimate the integrand. 
From  \eqref{deffi} and \eqref{eulerf}, 
\begin{align}
\nn
\log \big|\varphi^{}_{X(t)} (y)\big|=&
-\sum\limits_{\ell\ge 1}
\frac{2\,t^\ell(t^\ell +1)}{\ell\,(1-t^\ell)
\big((\frac{1-t^\ell}{\sin(y\ell/2)})^2+4\,t^\ell\big)}
\\ \nn
&\le -\,\frac{ 2\,t\,(t+1)}{(1-t) \big(a^{-2}\,\pi^2+4 t  \big)},
\end{align}  
where the sum is estimated by taking only its first term and then recalling that
the functions $\sin(x)$ and $\frac{x}{\sin(x)}$ are increasing for
 $x\in [0,\pi/2]$. Whence,  if $a(1-t)<|y|<\pi$, 
\[
\big(\frac{1-t}{\sin (y/2)}\big)^2<\big(\frac{a(1-t)/2}{ \sin(a(1-t)/2)\,a/2}\big)^2
 \le (\frac \pi a)^2.
\]
Finally, take $t=t_n$ in the above estimates and  recall that, as observed in \eqref{tnge}, 
   $t_n>\e^{-\frac\pi 2}$ for all  $n \in \mathbb N$  to conclude
\[
\big|\! \int\limits_{\mathcal A_2 }
  \bb E\, \e^{i\theta\,Z(t_n)}\, d\theta\,\big|\le 2 \pi\sigma(t_n)\,\e^{-\frac{C}{1-t_n}}, 
  \mbox{ where } C=2\,\e^{-\frac\pi 2}\frac{(1+\e^{-\frac\pi 2})}{(a^{-2}\pi^2+4)},
  \]
which, together with \eqref{sigma2a}, \eqref{tnle}  and \eqref{tna} implies that, 
for some positive  constants $C_1$ and $C_2$ we have  
\begin{equation}
\label{tailint}
\big|\int\limits_{\mathcal A_2}
 \!\bb E\, \e^{i\theta\,Z(t_n)}\, d\theta\,\big|\le C_1\e^{-C_2\sqrt n}.
\end{equation}
To  consider the integral over $\mathcal A_1$ we expand the integrand in \eqref{i2}  in terms
of the cumulants. 
To that end, observe that, for $t\in (0,1)$, we may write  
\begin{equation}
\label{phicum}
\bb E\, \e^{i\theta\,Z(t)}=\varphi^{}_{Z(t)}(\theta)=\exp{\big(K_{Z(t)}(\theta)\big)},
\end{equation}
with $ K_{Z(t)}$ the cumulant generating function of $Z(t)$ defined in the previous section. 
It is clear from the definition  \eqref{zt} that the first cumulant of
$Z(t)$ is zero, while the cumulants of order $j$ for  $ j\ge 2$ are $\kappa_j(t)/\sigma(t)$, 
where $\kappa_j (t)$ are the cumulants of the random variable $X(t)$ 
given by \eqref{rcum1j}. For each $M>2$, we consider the Taylor expansion of  order  $M$   for $K_{Z(t)}$.   
 From Lemma  \ref{lemaderK}, 
\begin{equation}
\label{KZm}
K_{Z(t)}(\theta)=
\sum_{j\ge 2}^M \kappa_j(t)\, \frac{(i\theta)^j}{j!\,\sigma(t)^j}+\mathcal{R}^{(1)}_{M+1},
\end{equation}
where the remainder $\mathcal{R}^{(1)}_{M+1}=\mathcal{R}^{(1)}_{M+1}(\theta,t)$
 can be estimated with the aid of 
Corollary \ref{derKZ}: 
\begin{equation}
\label{Rm1}
\big|\mathcal{R}^{(1)}_{M+1}(\theta,t)\big|
\le
\frac{|\theta|^{M+1}}{(M+1)!}\, \frac{\kappa_{M+1}(t)}{\sigma(t)^{M+1}}.
\end{equation}
 We substitute  the expressions for the cumulants obtained in Proposition \ref{cumulants>2}
into \eqref{KZm}  to write
\begin{align}
\label{KZ}
K_{Z(t)}\,(\theta)=-&\,\frac{\theta^2}{2}\,+\,\sum_{j=3}^{M} 
\big(\frac{i\theta}{\sigma(t)}\big)^j 
\,\big(\frac{\pi^2}{6|\log t|^{j+1}}-\frac{1}{2\,j|\log t|^{j}}\big)
\\ \nn
 +&\,\mathcal{R}^{(1)}_{M+1}+
\mathcal{R}^{(2)}_{M+1},
\end{align}
where 
\begin{equation} 
\label{defrm2}
\mathcal{R}^{(2)}_{M+1}=\mathcal{R}^{(2)}_{M+1}(\theta,t):=
\sum _{j=3}^M \big(\frac{i\theta}{\sigma(t)}\big)^j
\,\frac{E_j(t)}{j!}.
\end{equation}
If $\frac{|\theta|}{\sigma(t)|\log t|}<1$  we deduce from \eqref{Ejle} that  there is
 a positive  constant $\tilde C_M$ that may depend on $M$ such that, if $t>\e^{-\frac \pi 2}$, 
\begin{equation}
\label{rm2le}
|\mathcal{R}^{(2)}_{M+1}(\theta,t)|\le \tilde C_M 
\frac{\e^{-\frac{4\pi^2}{|\log t|}}}{|\log t|^M},  
\end{equation}
and that 
 the series resulting 
from taking $M=\infty$ in the sum in  \eqref{KZ} converges  to 
\begin{align}
\nn
\sum_{j\ge 3}& \big(\frac{i\theta}{\sigma(t)}\big)^j 
\,\Big(\frac{\pi^2}{6|\log t|^{j+1}}-\frac{1}{2\,j|\log t|^{j}}\Big)\,=
\\ \nn
&\frac{-\theta^2}{2}\Big(\frac{\pi^2}{3\,|\log t|^3\sigma(t)^2}
\frac{\frac{i\theta}{\sigma(t) |\log t|}}{(1-\frac{i\theta}{\sigma(t) |\log t|})}\Big)
\\ \nn
&\qquad +\frac{1}{2}\,\frac{i\theta}{\sigma(t) |\log t|} \,+
\frac{1}{4}\,\big(\frac{i\theta}{\sigma(t) |\log t|}\big) ^2 
+\,\frac{1}{2}\log\big(1- \frac{i\theta}{\sigma(t) |\log t|}\big).
\end{align}
Thus, from \eqref{KZ} we may write, for $\frac{|\theta|}{\sigma(t)|\log t|}<1$:
\begin{multline}
\label{KZinf}
K_{Z(t)}\,(\theta)=-\,\frac{\theta^2}{2}\,\Big(\,1+\frac{\pi^2}{3|\log t|^3\sigma(t)^2}\,
\frac{\frac{i\theta}{\sigma(t) |\log t|}}{(1-\frac{i\theta}{\sigma(t) |\log t|})}
+\,\frac {1}{2\sigma(t)^2 |\log t|^2 }\,\Big)
\\
+\frac{1}{2}\,\frac{i\theta}{\sigma(t) |\log t|} \,+
\,\frac{1}{2}\log\big(1- \frac{i\theta}{\sigma(t) |\log t|}\big)
+\sum^3_{\ell=1} \mathcal{R}^{(\ell)}_{M+1}
\end{multline}
where 
\begin{equation}
\label{defrm3}
\mathcal{R}^{(3)}_{M+1}=\mathcal{R}^{(3)}_{M+1}(\theta,t)=
-\sum_{j\ge M+1} \big(\frac{i\theta}{\sigma(t)}\big)^j 
\,\big(\frac{\pi^2}{6|\log t|^{j+1}}-\frac{1}{2\,j|\log t|^{j}}\big).
\end{equation}
It follows that for some constant $C>0$ that depends only on $a$, 
\begin{equation}
\label{rm3le}
|\mathcal{R}^{(3)}_{M+1}(\theta,t)|\le C \,\big(\frac{|\theta|}{\sigma(t)}\big)^{M+1}
\frac{1}{|\log t|^{M+2}}.
\end{equation}
Observe next that if $\theta \in \mathcal A_1$, then $\frac{|\theta|}{\sigma(t_n)|\log t_n|}<a<1$, so
\eqref{KZinf} is valid when evaluated at  $\theta \in \mathcal A_1$ and $t=t_n$, as well as the estimates
\eqref{Rm1}, \eqref{rm2le} and \eqref{rm3le}.
Recall the definition of $c_n$  \eqref{defcn}, and let us  denote
 \begin{equation}
 \label{defla}
\lambda_n= 1+\frac{1}{2c_n^2} =1+\frac{1}{2\,|\log t_n|^2\sigma^2(t_n)}, 
\end{equation} 
\begin{equation}
\label{defg12}
g_1(z)=\frac{z}{1-z},\qquad g_2(z)=\frac z2+\frac12 \log { (1-z)},
\end{equation}
to write, from \eqref{sigma2a} and   \eqref{KZinf}: 
\begin{align}
\label{Kintegrand}
\exp\big(K_{Z(t_n)}(\theta)\big)
=\,&\exp{\{
-\frac{\,\theta^2}{2}\lambda_n\, \big(1+\,g_1(i\theta/c_n)\big)
+g_2(i\theta/c_n)\}}
\\ \nn
& \times\,\exp{\{
\frac{\,\theta^2}{2}\,\epsilon_n\,g_1(i\theta/c_n)+\sum^3_{\ell=1}
 \mathcal{R}^{(\ell)}_{M+1}
\}}
\\ \nn
=\, &\exp{\{-\frac{\,\theta^2}{2}\lambda_n\, \big(1+\,g_1(i\theta/c_n)\big)
+g_2(i\theta/c_n)\}}
\\ \nn
&\times\,\big(\exp{\{
\frac{\,\theta^2}{2}\,\epsilon_n\,g_1(i\theta/c_n)+
\sum^3_{\ell=1} \mathcal{R}^{(\ell)}_{M+1}\}}-1\big)
\\ \nn
&\,+\,\exp{\{-\frac{\,\theta^2}{2}\,\lambda_n\, \big(1+\,g_1(i\theta/c_n)\big)
+g_2(i\theta/c_n)\}},
\end{align}
for  
\begin{equation}
\label{defeps}
\epsilon_n=\frac{E_2(t_n)}{\sigma^2(t_n)},\qquad |\epsilon_n|
\le C\,\frac{ \e^{-\frac{4\pi^2}{|\log t_n|}}}{|\log t_n|}\le C\sqrt n\,\e^{-2\pi\sqrt{24n-1}},
\end{equation}
for some positive constant $C$, as follows from \eqref{defE2}, \eqref{cjcj}, 
\eqref{tna} and \eqref{e4pi}.

We proceed now to obtain a suitable expansion for  the last line in \eqref{Kintegrand}
following the procedure in \cite{BM}. 
For each $N_0\ge 1$,  
\begin{align}
\nn
\e^{-\frac{\,\theta^2}{2}\,\lambda_n\,g_1(z)\,+g_2(z)}=1+
\sum_{j=1}^{N_0} \frac{z^j}{j!}\,\partial_z^{(j)}
\e^{-\frac{\,\theta^2}{2}\,\lambda_n\,g_1(z)+g_2(z)}\vert_{z=0}
+\,\mathcal{R}^{(4)}_{N_0+1}(z). 
\end{align}
In the case $z=iy$, $y\in \mathbb R$ the remainder satisfies 
\begin{equation}
\label{r4N}
|\mathcal{R}^{(4)}_{N_0+1}(z)|\le \frac{|z|^{N_0+1}}{(N_0+1)!}\,\big|\partial_z^{(N_0+1)}\,
\e^{-\frac{\,\theta^2}{2}\,\lambda_n\,g_1(z)+g_2(z)}\big\vert_{z=i\xi}\big|,\quad |\xi|<|y|.
\end{equation}
Substitute the above expansion with $z=i \theta/ c_n$ in \eqref{Kintegrand} and integrate over
 $\mathcal A_1$ to obtain 
 that for each $N_0 \ge 1$ 
\begin{multline}
\label{If}
 \int_{\mathcal A_1}
\bb E\, \e^{i\theta\,Z(t_n)}\, d\theta = \int_{\mathcal A_1}
\exp\big(K_{Z(t_n)}(\theta)\big)\, d\theta \\
 =\!\int _{\mathcal A_1}\!
\e^{-\frac{\,\theta^2}{2}\lambda_n}\,
\Big( 1+\, \sum_{j= 1}^{N_0} \frac{(i\theta/c_n)^j}{j!}\partial_z^{(j)}
\e^{-\frac{\theta^2}{2}\,\lambda_n\,g_1(z)+g_2(z)}\,
\big\vert_{z=0}\,\Big)\,d\theta\,
\\ 
\phantom{123456}+\int_{\mathcal A_1}
\e^{-\frac{\,\theta^2}{2}\lambda_n}\,\mathcal{R}^{(4)}_{N_0+1}(i\theta/c_n)\,d\theta\,+
\mathcal E_{M+1}^{(1)},
\end{multline}
where $\mathcal E_{M+1}^{(1)}$ is the error term arising from \eqref{Kintegrand}:
\begin{multline}
\label{defem1}
\mathcal E_{M+1}^{(1)}(n)=\int _{\mathcal A_1}
 \exp{\{-\frac{\,\theta^2}{2}\lambda_n\, \big(1+\,g_1(i\theta/c_n)\big)
+g_2(i\theta/c_n)\}}
\\ 
\times\,\big(\exp{\{
\frac{\,\theta^2}{2}\,\epsilon_n\,g_1(i\theta/c_n)+\sum^3_{\ell=1} 
\mathcal{R}^{(\ell)}_{M+1}
\}}-1\big)\,d\theta.
\end{multline}

It is easy to see that, for $j\in \mathbb N$ and $z\in \mathbb C$, 
\begin{equation}
\label{remder}
\partial_z^{(j)}\,
\e^{-\frac{\,\theta^2}{2}\,\lambda_n\,g_1(z)+g_2(z)}=
\e^{-\frac{\,\theta^2}{2}\,\lambda_n\,g_1(z)+g_2(z)}P_{j,z}(\lambda _n\,\theta^2),
\end{equation}
where $P_{j,z}(\cdot)$ is a polynomial of degree $j$ with coefficients depending on products 
of derivatives (up to order $j$) of the functions $g_1(z)$ and $g_2(z)$. Those derivatives are 
directly seen to be bounded if $|z|<a$. We can write then the first integral 
in the right hand side of  \eqref{If} as the difference of the corresponding
 one on the whole line and that over 
$\mathcal A_1^c$. We denote the latter by $\mathcal I_n$ to obtain:
\begin{multline}
\label{iac}
\int _{\mathcal A_1}
\e^{-\frac{\,\theta^2}{2}\lambda_n}\,
\Big( 1+\,\sum_{j= 1}^{N_0} \frac{(i\theta/c_n)^j}{j!}\partial_z^{(j)}
\e^{-\frac{\theta^2}{2}\,\lambda_n\,g_1(z)+g_2(z)}
\big\vert_{z=0}\,\Big)d\theta\,=
\\
\int_{\mathbb R}
\e^{-\frac{\,\theta^2}{2}\lambda_n}\,\Big(1+\,
\sum_{j= 1}^{N_0} \frac{(i\theta/c_n)^j}{j!}\partial_z^{(j)}
\e^{-\frac{\theta^2}{2}\,\lambda_n\,g_1(z)+g_2(z)}
\big\vert_{z=0}\,\Big)d\theta\,
- \mathcal I_n.
\end{multline}
Observe that all the integrals above  are finite. 
Moreover, we can compute the integral in $\mathbb R$. Indeed, recall that
 $\lambda_n(1+g_1(z))= \lambda_n/(1-z)$ and $\e^{\,g_2(z)}=\e^{z/2}\sqrt {1-z}$.  
Substitute into the integrand in \eqref{iac}, interchange the derivative
 with the integral and integrate term by term using that 
 \[
\int_{\mathbb R} \e^{-\frac{\,\theta^2}{2}\big(\lambda_n/(1-z)\big)}\,\theta^j\,d\theta=
\begin{cases}\quad  0 &\mbox{ if $j$ is odd }\\
\sqrt {2\,\pi}\,(j-1)!!\,(\frac{1-z}{\lambda_n})^{(j+1)/2}  &\mbox { if $j$ is even}.
\end{cases}
\]
Take, for the given $N>0$, $N_0=2N+1$ to obtain from \eqref{If} and \eqref{iac} 
after renaming the terms in the sum that 
\begin{multline}
\int _{\mathcal A_1}
  \bb E\, \e^{i\theta\,Z(t_n)}\, d\theta=
  \\
  \sqrt {\frac{2\pi}{\lambda_n}}\,\Big( 1+\,
\sum_{\ell= 1}^{N} (-1)^\ell\,\frac{1}{(\lambda_n\,c_n^2)^\ell}\,
\frac{1}{2^\ell\,\ell!}\,
\partial^{(2\ell)}_z\big(\e^{z/2}(1-z)^{\ell+1}\big)\big\vert_{z=0}\Big)
\\
 \label{ppalint}
+\mathcal E_{M+1}^{(1)}+ \mathcal E_{2(N+1)}^{(2)},
\end{multline}
where 
\begin{equation}
\label{defem2}
\mathcal E_{2(N+1)}^{(2)}(n)=\int_{\mathcal A_1}
\e^{-\frac{\,\theta^2}{2}\lambda_n}\,\mathcal{R}^{(4)}_{2(N+1)}(i\theta/c_n)d\theta
- \mathcal I_{n}.
\end{equation}
Observe now that from  \eqref{defla} and Lemma \ref{lemcn}   
\begin{equation}
\nn
\lambda_n\,c_n^2=c_n^2 +\frac 12 =
\sqrt{\frac{2\pi^2}{3}(n-\frac{1}{24})+\frac 14}+\frac 12 +\gamma_n,
\end{equation}
with $|\gamma_n| \le\,C\, n\,{\rm e}^{-2\,\pi\sqrt{24n-1}}$. 
The derivatives in \eqref{ppalint} are easily computed: for $ \ell\ge 1  $,  
\begin{multline}
\label{defCl}
J_\ell:=\partial^{(2\ell)}_x\big(\e^{x/2}(1-x)^{\ell+1}\big)\big\vert_{x=0}\\
=\frac{(\ell+1)!}{4^\ell}\sum_{k=0}^{\ell+1}(-1)^k\, 2^k\,\binom{2\ell}{k}
\frac{1}{(\ell+1-k)!}.  
\end{multline}
Thus,  substitution into \eqref{ppalint}, after expressing $\lambda_n$ in terms of 
$c_n$ and clearing up 
yields the expansion on the right hand side of \eqref{Thint}. From  \eqref{tailint},  
to finish the proof of Theorem \ref{Tint} it suffices to show that: 
\begin{itemize}
\item[a)] It is possible to choose $M>2$ such that $\mathcal E_{N+1}:=
 \mathcal E^{(1)}_{M+1}+\mathcal E^{(2)}_{2(N+1)}$ satisfies \eqref{Eint}. 
\item[b)] $0<D_\ell<2\,3^{\ell-1} \quad \forall \ell$. 
\end{itemize}

{\it Proof of }a)\quad We estimate each of the terms in the last line of
 \eqref{ppalint}, starting with $\mathcal E_{M+1}^{(1)}$,  that was  defined
  in \eqref{defem1}. 
Observe first that if $\theta \in \mathcal A_1$,
\begin{align}
 \label{defL1}
L_1:=&\big|\exp{\{-\frac{\,\theta^2}{2}\,\lambda_n\big(1+g_1(i\theta/c_n)\big)
+g_2(i\theta/c_n)\}} \big|
\\ \nn
=&\exp{\{-\frac{\,\theta^2}{2}\lambda_n\,\Re\big(1+ g_1(i\theta/c_n) \big)
+\,\Re\big( g_2(i\theta/c_n) \big)\}}
  \le  \exp{\{-\frac{\,\theta^2}{4}\}}.  
\end{align}
The  inequality follows from  observing that
 $\lambda_n> 1$  and that, for $\theta \in \mathcal A_1$,   
$|\frac{\theta}{c_n}|<\frac{|\theta|}{\sigma(t_n)(1-t_n)}<a<1$ so from \eqref{defg12} 
 \begin{align}
\nn
&\Re\big(1+g_1(i\theta/c_n)\big)=\frac{1}{1+\theta^2/c_n^2}\,> \frac12,
\\ \nn
&\Re\big(g_2(i\theta/c_n)\big)=\sum_{k\ge1}(-1)^k\big(\theta/c_n)^{2k}\frac{1}{4k}\,< 0.
\end{align}
To estimate the errors that compose $\mathcal E_{N+1}$,  we use the following inequality, that is a direct consequence of 
Proposition \ref{cumulants>2} and \eqref{Ejle}: for each $M$, 
 there is a positive constant $C_{M}$ that may depend on $M$ such that
\begin{equation}
\label{km1}
\frac{\kappa_{M+1}(t_n)}{(M+1)!}\le \frac{C_M}{|\log t_n|^{M+2}}.
\end{equation}
Next, we choose $\alpha \in (1,\frac 32 )$ and decompose
the region of  integration $
\mathcal A_1$ as the disjoint union 
$
\mathcal A_1 =\mathcal A_{1,1}\cup \mathcal A_{1,2}, 
$
for 
\[
\mathcal A_{1,1}=\{\theta: \frac{|\theta|}{\sigma(t_n)}<|\log t_n|^\alpha\}.
\]
For $\theta\in \mathcal A_{1,1} $ we obtain  from \eqref{Rm1} and \eqref{rm3le}, 
with the aid of  \eqref{km1}:
\begin{align}
\nn
\big|\mathcal{R}^{(1)}_{M+1}(\theta,t_n)\big|&\le
\big(\frac{|\theta|}{\sigma(t_n)}\big)^{M+1}\,
 \frac{\kappa_{M+1}(t_n)}{(M+1)!}
  \le C_M\,\frac{ |\log t_n|^{\alpha(M+1)}}{ |\log t_n|^{(M+2)}},
 \\ \nn
\big|\mathcal{R}^{(3)}_{M+1}(\theta,t_n)\big|&\le C\,  \big(\frac{|\theta|}{\sigma(t_n)}\big)^{M+1}
\frac{1}{|\log t_n|^{M+2}}
\le C\,\frac{ |\log t_n|^{\alpha(M+1)}}{ |\log t_n|^{(M+2)}}.
\end{align}
From  these inequalities plus \eqref{rm2le} evaluated at $t=t_n$  we conclude that, given $N>0$, 
  if $\theta \in \mathcal A_{1,1}$, 
we can take $M$ sufficiently large such that there is a constant $\overline{C}>0$
\begin{equation}
\label{rm11}
|\mathcal{R}^{(\ell)}_{M+1}(\theta,t_n)|\le\,\overline{C} |\log t_n|^{N+1},\quad \ell=1,2,3.
\end{equation}
In the case  
 $\theta\in \mathcal A_{1,2} $, again  from \eqref{Rm1} and
  \eqref{rm3le} and with the aid of \eqref{km1} we write
\begin{align}
\nn
|\mathcal{R}^{(1)}_{M+1}(\theta,t_n)|&\le |\theta|^2 a^{M-1}|\log t_n|^{M-1}
\frac{\kappa_{M+1}(t_n)}{(M+1)!\,\sigma^2(t_n)}
\\\label{rm1da2}
&\le |\theta|^2 a^{M-1}\,\frac{C_M}{\sigma^2(t_n)|\log t_n|^3}
\le \beta_1 |\theta|^2,
\\
\label{rm3da2}
|\mathcal{R}^{(3)}_{M+1}(\theta,t_n)|&\le C\,|\theta|^2\,a^{M-1}
\frac{1}{\sigma^2(t_n)|\log t_n|^3}\le \beta_2 |\theta|^2. 
\end{align}
Recall that $\sigma^2(t)=\kappa_2(t)$ and that  from \eqref{sigma2}, 
 $\sigma^2(t_n)|\log t_n|^3 \to\pi^2/3$ as $n\to \infty$, thus  the 
 ${\beta_j}$'s, $j=1,2$ can be taken smaller that $1/10$ for sufficiently large $n$ 
 by choosing $M$ large.
 In addition, for $\theta \in \mathcal A_1$, from  the definition \eqref{defeps} of $\epsilon _n$, 
 from \eqref{defg12} and using the estimate \eqref{Ejle} for $E_2(t_n)$ we conclude that, for each 
 given $N>0 $ there is a positive  constant $C_N$ such that 
 \begin{equation}
 \nn
 |\theta^2\,\epsilon _n\, g_1(i\theta/c_n)|=|\theta^2\frac{E_2(t_n)}{\sigma^2(t_n)} g_1(i\theta/c_n)|
 \le E_2(t_n)|\log t_n|^2\le C_N|\log t_n|^{N +1}.
 \end{equation}
Then, for $n$ large enough, from \eqref{rm11}, \eqref{rm1da2}, \eqref{rm3da2}, \eqref{rm2le},
 the last inequality above and \eqref{defL1} 
we estimate $\mathcal E_{M+1}^{(1)}$ as follows:
\begin{align}
\nn
|\mathcal E_{M+1}^{(1)}|\le &\int\limits_{\mathcal A_{1,1}}
 L_1
\times\,\big|\exp{\{
\frac{\,\theta^2}{2}\,\epsilon_n\,g_1(i\theta/c_n)+\sum^3_{\ell=1} 
\mathcal{R}^{(\ell)}_{M+1}
\}}-1\,\big|\, d\theta
\\ \nn
&+ \int\limits_{\mathcal A_{1,2}}
 L_1
\times\,\big|\exp{\{
\frac{\,\theta^2}{2}\,\epsilon_n\,g_1(i\theta/c_n)+\sum^3_{\ell=1} 
\mathcal{R}^{(\ell)}_{M+1}
\}}-1\,\big|\, d\theta 
\\ \nn
 \le &\int\limits_{\mathcal A_{1,1}}
\e^{-\frac{\theta^2}{4}}\,\big|
\frac{\,\theta^2}{2}\,\epsilon_n\,g_1(i\theta/c_n)+\sum^3_{\ell=1} 
\mathcal{R}^{(\ell)}_{M+1}
\}\big|\, d\theta 
\\ \nn
&+\int\limits_{\mathcal A_{1,2}} \e^{-\frac{\theta^2}{4}}\,
\big(\e^{\frac{\theta^2}{5}+\frac{1}{10}}+1\big)\,d\theta
\\ \nn
 \le & (C_N+\overline C)\,|\log t_n|^{N+1}\! \int\limits_{\mathcal A_{1,1}}
\e^{-\frac{\theta^2}{4}}\,d\theta 
\\ \nn
&+ \e^{-\frac{\sigma^2(t_n)\,|\log t_n|^{2\alpha}}{40}}\!
\int\limits_{\mathcal A_{1,2}} \e^{-\frac{9\,\theta^2}{40}}\big(\e^{\frac{\theta^2}{5}+\frac{1}{10}}+1\big)
d\theta.
\end{align}
Since the above integrals are finite and  $\sigma^2(t_n)|\log t_n|^{2\alpha}=O(|\log t_n|^{2\alpha-3})$
 with  $\alpha <3/2$,  
we conclude with the aid of \eqref{ltnrn}   that, given $N>0$, there exists $M>2$ such that, 
for sufficiently large $n$ there are positive constants $\overline  {C}_j$, $j=1,2$  such that  
 \begin{equation}
 \label{estEM1}
 |\mathcal E_{M+1}^{(1)}|\le \overline {C}_1|\log t_n|^{N+1}\le \overline C_2\, n^{-(\frac{N+1}{2})}.
 \end{equation}
 It remains to consider $\mathcal E^{(2)}_{2(N+1)}$,  defined in \eqref{defem2}.
 
 For $N> 0$, from \eqref{r4N}, \eqref{remder}  and the remark below,
  using  \eqref{defL1}  we obtain
 \begin{equation}
 \nn
 \e^{-\frac{\theta^2}{2}\,\lambda_n} \big|\mathcal R_{2(N+1)}^{(4)}(\frac{i\theta}{c_n})\big|\le \frac{1}{c_n^{2(N+1)}}\,
 \frac{|\theta|^{2(N+1)}}{(2N+2)!}\,\e^{-\frac{\theta^2}{4}}\,|P_{2(N+1),\xi}(\lambda_n\theta^2)|.
 \end{equation}
It follows then that, for some polynomial $Q$  with bounded coefficients
 (depending on the derivatives of $g_1$ and $g_2$ over  $[-a,a]$), 
\begin{equation}
\nn
 \int_{\mathcal A_1}
 \e^{-\frac{\theta^2}{2}\,\lambda_n}\big|\mathcal R_{2(N+1)}^{(4)}(\frac{i\theta}{c_n})\big|\,d\theta \le
 \frac{1}{c_n^{2(N+1)}}\,
 \int_{\mathcal A_1}
 \e^{-\frac{\theta^2}{4}}\,\big|Q(\theta^2)\big|d\theta
 \end{equation}
 and, for some constant $C_I>0$, 
 \begin{multline}
 \nn
|\mathcal I_n|=\big|\int _{\mathcal A_1^c}
\e^{-\frac{\,\theta^2}{2}\lambda_n}\,\Big(1+\,
\sum_{j= 1}^{2N} \frac{(i\theta/c_n)^j}{j!}\partial_z^{(j)}
\e^{-\frac{\theta^2}{2}\,\lambda_n\,g_1(z)+g_2(z)}
\big\vert_{z=0}\,\Big) d\theta\,\big|
\\
\le \int _{\mathcal A_1^c}
\e^{-\frac{\,\theta^2}{2}\lambda_n}\,\Big|1+\,
\sum_{j= 1}^{2N} \frac{(i\theta/c_n)^j}{j!} P_{j,0}(\lambda_n\theta^2)\Big| d\theta\,
\le C_I\,\e^{-\frac{\sigma^2\,(t_n)\,|\log t_n|^2}{4}}. 
 \end{multline}
Since $\sigma^2\,(t_n)\,|\log t_n|^2 =O(|\log t_n|^{-1})$, it  follows from the above 
estimates that given $N>0$, 
\[
|\mathcal E^{(2)}_{2(N+1)}|=O(c_n^{-2(N+1)})=O(n^{-(N+1)/2})
\]
 what concludes the proof of a).
 
 \smallskip 
{\it Proof of }b)\quad
We proceed to  show that the  ${J_\ell}$'s defined in \eqref{defCl}   alternate sign. 
Define,  for each fixed $\ell$,  
$a_k=2^k\binom{2\ell}{k}\frac{1}{(\ell+1-k)!}$, 
and $d_k=a_k-a_{k+1}$. It is easy to verify that $d_k\le 0$ for any $\ell\ge 1$ and 
$k\le\ell$ and then,
\begin{equation}
\nn
J_\ell =\begin{cases}
\frac{(\ell+1)!}{4^\ell}
\sum_{\overset{k\ge 0,}{k \text { even}}}\limits^{\ell}d_k &\mbox{ if  $\ell$ is even,} 
\\
\phantom{}
\\
\frac{(\ell+1)!}{4^\ell}
\big(\frac{1}{(\ell+1)!}+\sum_{\overset{k\ge 1,}{ k\text { odd}}}\limits^{\ell} -d_k\big)
&\mbox{  if  $\ell$ is odd.}
\end{cases} 
\end{equation}
In particular, the last formula implies that  $(-1)^\ell J_\ell\le 0$,
 and from the identity in \eqref{defdl}, we have  
 \begin{equation}
 \label{ClDl}
 D_\ell=(-1)^{\ell+1}\frac{J_\ell}{\ell!},
 \end{equation}
 thus $D_\ell >0$. 
Let us show now that $D_\ell<2\,3^{\ell-1}$. 
 Since $d_k \le 0$  it follows  that if $\ell$ is odd,  
\[
0\le J_\ell = \frac{(\ell+1)!}{4^\ell}\big(\sum _{k=0}^\ell (-1)^k\,a_k +a_{\ell+1}\big)
\le 
 \frac{a_{\ell+1}(\ell+1)!}{4^\ell}=\frac{2^{\ell+1} (2\ell)!}{4^\ell(\ell-1)!},
 \]
and if $\ell$ is even
\[
0\ge J_\ell = \frac{(\ell+1)!}{4^\ell}\big(\sum _{k=0}^\ell (-1)^k\,a_k -a_{\ell+1}\big)\ge 
 -\frac{a_{\ell+1}(\ell+1)!}{4^\ell}=-\frac{2^{\ell+1} (2\ell)!}{4^\ell(\ell-1)!},
 \]
which is summarized as 
\[
|J_\ell|\le \frac{(2\ell)!}{2^{\ell-1}(\ell-1)!},
\] 
and from \eqref{ClDl}
\[
D_\ell\le \frac{(2\ell)!}{2^{\ell-1}(\ell-1)!\ell!}=2\,\frac{(2\ell-1)!!}{(\ell-1)!}\le 2\,3^{\ell-1}, 
\]
what concludes the proof of b), and that of Theorem \ref{Tint}.
\end{proof}

\begin{remark}
The inequalities \eqref{Dele} in particular imply that the sum in \eqref{Thint}
 is convergent (as $N\to \infty$)
 for any  $n$ such that $c_n^2>1$, which is the case for any $n\ge 1$,
  as follows for instance from \eqref{estdn}. 
  However, \eqref{Thint} is an asymptotic expansion, and  has to be interpreted as such:
   for fixed $N$, it approximates the integral on the left hand side 
  for $n$ large, with an error satisfying \eqref{Eint}. 
  \end {remark}
  \section{Proof of the main results  and a table} 
  \label{finale}
  We present here the proofs of Theorem \ref{Thpp} and  Proposition \ref{propalt}. 
  The expansion obtained in this last result is computed for several values of $n$, 
  taking $N=17$, and compared with the true value of $p(n)$, as reported in Table \ref{table}. 
\begin{proof} [Proof of Theorem \ref{Thpp}]
 From \eqref{pn} and Corollary \ref{Corfactor}, 
we know that   
\begin{align}
\label{mult}
p(n)=&\frac{|\log t_n|^{1/2}}{(2\pi)^{3/2}\,\sigma(t_n)}\,
{\rm e}^{\sqrt{\frac {2\pi^2}{3}\,(n-\frac{1}{24})+\frac 14}\,+
 \beta_n}\,
 \int^{\pi\sigma(t_n)}_{-\pi\sigma(t_n)}
\bb E\, \e^{i\theta\,Z(t_n)}\, d\theta 
\\ \nn
=&
\frac{|\log t_n|^{1/2}}{(2\pi)^{3/2}\,\sigma(t_n)}\,
{\rm e}^{\sqrt{\frac {2\pi^2}{3}\,(n-\frac{1}{24})+\frac 14}}\,
 \int^{\pi\sigma(t_n)}_{-\pi\sigma(t_n)}
\bb E\, \e^{i\theta\,Z(t_n)}\, d\theta
\\ \nn
&+  \big({\rm e}^{\beta_n}-1\big)\,\frac{|\log t_n|^{1/2}}{(2\pi)^{3/2}\,\sigma(t_n)}\,
{\rm e}^{\sqrt{\frac {2\pi^2}{3}\,(n-\frac{1}{24})+\frac 14}}\,
 \int^{\pi\sigma(t_n)}_{-\pi\sigma(t_n)}
\bb E\, \e^{i\theta\,Z(t_n)}\, d\theta.
 \end{align}
 \nn
 Let us denote by $\mathcal Q_n$ the last line above. 
For $n$ sufficiently large we can estimate  with  the aid of \eqref{betale}, after recalling that
 $|\bb E\, \e^{i\theta\,Z(t_n)}|\le 1$ and  
that $|\log t_n|\to 0 $ as $n \to \infty$,   
\begin{align}
 |\mathcal Q_n|&\le \big|{\rm e}^{\beta_n}-1\big|\, \frac{|\log t_n|^{1/2}}{(2\pi)^{1/2}}
 {\rm e}^{\sqrt{\frac {2\pi^2}{3}\,(n-\frac{1}{24})+\frac 14}}
 \\ \nn
 &\le {\rm e}^{-2\pi\sqrt{24\,n-1}} \,
  {\rm e}^{\sqrt{\frac {2\pi^2}{3}\,(n-\frac{1}{24})+\frac 14}}
  \le {\rm e}^{-\,(2\pi-1)\sqrt{24\,n-1}}.
 \end{align}
Substituting the integral by the expansion obtained in Theorem \ref{Tint}  we 
have that the second line in \eqref{mult}  equals 
\begin{equation}
\nn
\frac{|\log t_n|^{1/2}}{(2\pi)^{3/2}\sigma(t_n)}\,{\rm e}^{\sqrt{\frac {2\pi^2}{3}
\,(n-\frac{1}{24})+\frac 14}}\,
\Big(\sqrt{2\pi}\,\sqrt{\frac{2\,c_n^2}{1+2\,c_n^2}}
\big(1-\sum_{\ell= 1}^N\,\frac{D_\ell}{(1+2\,c_n^2)^\ell }\big)
+\mathcal E_{N+1}\Big).
\end{equation} 
Thus, \eqref{asymexp}  follows after substituting the last expression into 
\eqref{mult}, collecting terms after expressing $\sigma(t_n)=c_n/|\log t_n|$   
from \eqref{defcn} and denoting by   $\widetilde{\mathcal R}_{N+1}$
 the terms coming from  $\mathcal E_{N+1}$
and  $\mathcal Q_n$, that are directly seen to have the required order. Observe that the coefficients 
 $D_\ell$ are those computed in Theorem \ref{Tint}, and the stated inequalities are precisely 
 \eqref{Dele}. 
\end{proof}
\begin{proof}
[Proof of Proposition \ref{propalt}] 
For each $N>0$, let us estimate the difference $\Lambda_N$  of the sums in \eqref{asymexpa} and \eqref{asymexp}
as follows: 
\begin{multline}
\nn
\big|\Lambda_N\big|=\big|\sum _{\ell=1}^N\frac{D_\ell}{(1+2\,c_n^2)^\ell}-
\frac{D_\ell}{(1+2\,r_n)^\ell}\big|
=\big|\sum _{\ell=1}^N\frac{D_\ell}{(1+2\,r_n)^\ell}\{1-\big(\frac{1+2\,r_n}{1+2\,c_n^2}\big)^\ell\}\big|
\\
\le |1-q_n|\sum _{\ell=1}^N\frac{D_\ell}{(1+2\,r_n)^\ell}\{1+q_n+q_n^2\cdots+q_n^{\ell-1}\}\le 
|1-q_n|\, \sum _{\ell=1}^N\big(\frac{1}{4}\big)^\ell\,\ell\, D_\ell
\\
\le 2\,|c_n^2-r_n| \,\sum _{\ell=1}^\infty(\frac 14 )^\ell\ell \,D_\ell \le C_1\, 
 n\,{\rm e}^{-2\,\pi\sqrt{24n-1}},
\end{multline}
where we have denoted $q_n=\frac{1+2\,r_n}{1+2\,c_n^2}$ and  used that $1+2r_n> 6$ for all $n$;
we have also assumed  that $n$ is large enough 
so that $q_n\le \frac 32$ (recall $q_n\to 1$) so the series is convergent,
as can be seen from \eqref{Dele}.    The last inequality is a consequence of
    \eqref{estdn}, with $C_1$  a positive constant. 
    
To estimate the difference $\Gamma_n$ of the factors in \eqref{asymexpa} and \eqref{asymexp}, 
observe first  that, after simple computations using \eqref{defcn} and  \eqref{sigma2a}
we have 
\[
\frac{|\log t_n|^{3/2}}{\sqrt 2 \, \pi(1+2\,c_n^2)^{1/2}}
=\frac{\sqrt 3\,|\log t_n|^2}{2\pi^2\sqrt{1+\rho_n}},\mbox{ with } \rho_n=\frac{3|\log t_n|^3\,E_2(t_n)}{\pi^2}.
\]
Thus, from \eqref{ltnrn}, and estimating $\rho_n$ with the aid of \eqref{defE2},
\eqref{cjcj} and
\eqref{e4pi}, we obtain that, for $C_3$ a positive constant,  
\begin{multline}
\nn
|\Gamma_n|=\big|\frac{|\log t_n|^{3/2}}{\sqrt 2 \, \pi\,(1+2\,c_n^2)^{1/2}}-\frac{2\pi^2}{3\sqrt 3 (1+2 r_n)^2}\big|=
\frac{\sqrt 3}{2\pi^2}\,\big|\frac{|\log t_n|^{2}}{\sqrt{1+\rho_n}}-(\frac{2\pi^2}{3\, (1+2 r_n)})^2\big|
\\ 
\le \frac{\sqrt 3}{2\pi^2}\,|\log t_n|^2\big|\frac{1}{\sqrt{1+\rho_n}}-1\big|+\frac{\sqrt 3}{2\pi^2}\,
\big|\,|\log t_n|-\frac{2\pi^2}{3\, (1+2 r_n)}  \big|
\\
\times\big|\,|\log t_n|+\frac{2\pi^2}{3\, (1+2 r_n)}\big|
\\
\le C_2  |\log t_n|^2\,|\rho_n|
+C \,{\rm e}^{-2\,\pi\sqrt{24n-1}} \le C_3 \,{\rm e}^{-2\,\pi\sqrt{24n-1}}. 
\end{multline}
Let us then  add and substract
$\frac{2\pi^2\,{\rm e}^{r_n}}{3\sqrt 3 (1+2 r_n)^2}
\big(1-\sum _{\ell=1}^N\frac{D_\ell}{(1+2\,r_n)^\ell}\big)$
to the expression \eqref{asymexp} for $p(n)$  to obtain
\begin{multline}
\nn
p(n)=\frac{2\pi^2\, {\rm e}^{r_n} }{3\sqrt 3 (1+2 r_n)^2}
\Big(1-\sum _{\ell=1}^N\frac{D_\ell}{(1+2\,r_n)^\ell} +\frac{3\sqrt 3 (1+2 r_n)^2}{2\pi^2}\times 
\\
\big(\frac{|\log t_n|^{3/2}}{\sqrt 2 \, \pi\,(1+2\,c_n^2)^{1/2}}\widetilde{\mathcal R}_{N+1}+
\Gamma_n\,\big(1-\sum _{\ell=1}^N\frac{D_\ell}{(1+2\,c_n^2)^\ell}\big)+ 
\frac{2\pi^2}{3\sqrt 3 (1+2 r_n)^2}\,
\Lambda_N\big)\Big).
\end{multline} 
To conclude, call $\mathcal R_{N+1}$ the expression following the first  sum above,  and observe that,
from  the above estimates for $\Lambda_N$ and $\Gamma_n$, it is easy to conclude
 that it has the same order as 
 $\widetilde{\mathcal R}_{N+1}$. \end{proof}
\begin{table}[h!]
   \resizebox{\textwidth}{!}{
   \begin{tabular}{|c|r|r|l|}
   \hline\noalign{\smallskip}
\multicolumn{1}{|c|}
 {$\mathbf{n}$} &  \multicolumn{1}{|c|}
  {$\mathbf{p(n)}$} & 
   \multicolumn{1}{|c|}{$\mathbf{\bar p(n)}$  ($N=17$)}&
    \multicolumn{1}{|c|}{$\mathbf{\bar p(n)/p(n)}$ }
  \\
  \hline
  10   & $42$ 							& 42 						& 1 \\
  11   & $56$ 							& 57 						& 1.0178571428571428571  \\ \hline
  50   & $204 \,226$ 					& $204 \,211$ 					& 0.9999265519571455152 \\
  51   & $239 \,943$ 					& $239 \,959$ 					& 1.0000666825037613100 \\ \hline
  100 & $190 \,569 \,292$				& $190 \,568 \,945$ 			& 0.9999981791400054107 \\
  101 & $214 \,481 \,126$				& $214 \,481 \,499$ 			& 1.0000017390807618196 \\ \hline
  200 & $3 \,972 \,999 \,029 \,388$ 		& $3 \,972 \,998 \,993 \,186$ 		& 0.9999999908879917331 \\ 
  201 & $4 \,328 \,363 \,658 \,647$ 		& $4 \,328 \,363 \,696 \,288$ 		& 1.0000000086963580162\\ \hline
  500 & $2 \,300 \,165 \,032 \,574 \,323 \,995 \,027$ & $2 \,300 \,165 \,032 \,573 \,762 \,997 \,377$ & 0.9999999999997561054\\ \hline
  600 & $458 \,004 \,788 \,008 \,144 \,308 \,553 \,622$ & $458 \,004 \,788 \,008 \,137 \,064 \,138 \,753$ & 0.9999999999999841826 \\ \hline
  700 & $60 \,378 \,285 \,202 \,834 \,474 \,611 \,028 \,659$ & $60 \,378 \,285 \,202 \,834 \,397 \,465 \,935 \,949$ & 0.9999999999999987223\\ \hline
  800 & $5 \,733 \,052 \,172 \,321 \,422 \,504 \,456 \,911 \,979$ & $5 \,733 \,052 \,172 \,321 \,421 \,800 \,242 \,439 \,308$ & 0.9999999999999998772\\ \hline
   900 & $415 \,873 \,681 \,190 \,459 \,054 \,784 \,114 \,365 \,430$ & $415 \,873 \,681 \,190 \,459 \,049 \,122 \,378 \,030 \,945$ & 0.9999999999999999863\\ \hline
  1000 & $24 \,061 \,467 \,864 \,032 \,622 \,473 \,692 \,149 \,727 \,991$ & $24 \,061 \,467 \,864 \,032 \,622 \,432 \,794 \,750 \,374 \,387$ & 0.9999999999999999983 \\
  \hline
\end{tabular} 
} 
 \caption{
 {\footnotesize Comparison of $\bar p(n)$ (the nearest integer to  the sum in 
 \eqref{asymexpa} with $N=17$) and the true value of $p(n)$. Computations made with the help of 
 Mathematica 9.0,  {Wolfram Research{,} Inc.}}}\label{table}
\end{table}

\appendix

\section {Formulae for higher order cumulants}
\label{sec:A}
In the next lemma, we derive a couple of explicit expressions for each
 cumulant $\kappa_j(t)$, in terms of 
the Eulerian polynomials $A_j(t)$, that are defined through the following identity
\begin{equation}
\label{defpEu}
\sum _{k\ge 0} k^j\,t^k =\frac{A_j(t)}{(1-t)^{j+1}}.
\end{equation}
The first  four  are: 
\begin{equation}
\label{pEu}
A_0(t)=1,\quad A_1(t)=t,\quad A_2(t)=t+t^2,\quad A_3(t)=t+4t^2+t^3.
\end{equation}
More details can be found in L. Comtet's book  \cite{com}. 
\begin{lemma}
\label{lemacum} The cumulants $\kappa_j(t)$ satisfy 
\begin{align}
\label{f1kapa}
 \kappa_j(t) \; &=  
 \; \sum_{\ell \geq 1} \,\frac{\ell^{j-1} \,A_{j}(t^\ell)}{(1-t^\ell)^{j+1}}  &j\ge 1,
 \\ \label{f2kapa}
 \kappa_j(t)\; &= \sum_{\ell \geq 1} \,\frac{\ell^j\,A_{j-1}(t^\ell)}{(1-t^\ell)^j}   &j\ge 2.
\end{align}
The above series are absolutely  convergent for $|t|<1$. 
\end{lemma}  
 \begin{proof} Denote by $S_j(t)$ the right hand side of \eqref{f1kapa}. 
In \eqref{cum12a} we have seen that   $\kappa_1(t)=\sum_{\ell\ge1}\frac{t^\ell}{(1-t^\ell)^2}$, which equals $S_1(t)$ by \eqref{pEu};
 that proves \eqref{f1kapa}
for $j=1$. To conclude,  from \eqref{rcum1j} it suffices to show that 
\[
t\partial_t S_j(t)=S_{j+1}(t) \mbox { if }  j\ge 1.
\]
Using a  well known recurrence relation  for the Eulerian polynomials
that can be found for instance in  \cite[p.292]{com}
\begin{equation}
\label{recEP}
A_{j+1}(t)=t\,(1-t)\,A'_j(t)+(j+1)\,t\,A_j(t),
\end{equation}
 we compute
\begin{align}
\nn
t\partial_t\,\frac{\ell^{j-1}\,A_{j}(t^\ell)}{(1-t^\ell)^{j+1}}&=\frac{\ell^j}{(1-t^\ell)^{j+2}}\,
\big(t^\ell(1-t^\ell)\,A_j'(t^\ell)\,+(j+1)\, t^\ell \,A_j(t^\ell)\big)
\\ \nn
&=\frac{\ell^j}{(1-t^\ell)^{j+2}}\,A_{j+1}(t^\ell).
\end{align}
Summing up in $\ell$  the last expression we obtain $S_{j+1}$, and conclude
 the proof of \eqref{f1kapa}. 
To prove \eqref{f2kapa}, observe that, from \eqref{f1kapa} and
  \eqref{defpEu}, if $j\ge2$, 
\begin{multline}
\nn
\kappa_j(t)=\sum_{\ell \geq 1} \,\frac{\ell^{j-1} \,A_{j}(t^\ell)}{(1-t^\ell)^{j+1}} 
=\sum_{\ell \geq 1} \ell^{j-1}\sum_{m\ge 0}m^j\,t^{\ell m}\\
=\sum_{m\ge 0}\sum_{\ell \geq 1}m^jt^{\ell m} \ell^{j-1}=
\sum_{m\ge 0}m^j\frac{A_{j-1}(t^m)}{(1-t^m)^j}.
\end{multline}
The convergence of the series in \eqref{f1kapa} and \eqref{f2kapa} 
is clear from the fact that the polynomials
$A_j(0)=0$ for any $j\ge 1$.  
\end{proof}
The next result is a consequence of Lemma \ref{lemacum} and a  reasoning similar to that 
leading to  prove Lemma \ref{cumulants}: from the recurrence \eqref{rcum1j}
 we derive functional equations
for $\kappa_j$ for $j>2$, that yield  asymptotic formulae as $t\uparrow 1$  
for those. 
The precise statement is given next.  
\begin{proposition}
\label{cumulants>2}
The cumulants $\kappa_j(t)$, $j\ge 2$ satisfy the following functional equations
\begin{equation}
\label{kj}
\kappa_j(t)=\frac{\pi^2\,j!}{{6\,|\log t|^{j+1}}}-\frac{(j-1)!}{2\,|\log t|^j} +E_j(t),
\end{equation}
where the  terms $E_j(t)$ are given by the following expression
\begin{equation}
\label{fEj}
 E_j(t) = \frac{(j-1)!}{|\log t|^j} \sum_{r=1}^j \binom{j}{r}
 \Big( \frac{-4 \pi^2}{|\log t|} \Big)^r \frac{\kappa_r 
 ({\rm e}^{-\frac{4\pi^2}{|\log t|}} )}{(r-1)!}, \quad  j\ge 2,
\end{equation}
and satisfy
\begin{equation}
\label{aEj}
E_j(t)\asymp {\rm e}^{-\frac{4\pi^2}{|\log t|}}.
\end{equation}
\end{proposition}
\begin{proof} Recall that  the $\kappa_j(t)$ are obtained from the
 recurrence \eqref{rcum1j}.
We already know  from Lemma \ref{cumulants} and
  Corollary \ref{cormsigma} that the proposition holds for $j=2$. 
Let us start from $\kappa_2(t)$ as given in \eqref{sigma2}, and denote 
$(t\partial_t)^{(k)}$ the $k$-th iteration of the operator $t\partial_t$.  
Use \eqref{fH} with $H$ the first two  terms
  in the right hand side of \eqref{sigma2}
to obtain  directly by induction that  
 for $j>2$, 
\[
(t\partial_t)^{(j-2)}(\frac{\pi^2}{3|\log t|^3}-\frac{1}{2|\log t|^2})=
\frac{\pi^2\,j!}{{6\,|\log t|^{j+1}}}-\frac{(j-1)!}{2\,|\log t|^j},
\]
 which are the first two terms on the right hand side of \eqref{kj}.
Thus, the terms $E_j$ also satisfy  the recurrence 
\[
E_{j+1}=t\partial_tE_{j} \quad j\ge 2. 
\]
To prove \eqref{fEj} for $j>2$ it is enough to verify   that
 the expression satisfies the above recurrence, 
which follows by a straightforward induction in $j$. 
   
Now, from \eqref{rcum1j} and \eqref{f2kapa},  
\[
\partial_\lambda\,\kappa_r(\lambda)=\frac{\kappa_{r+1}(\lambda)}{\lambda}=
 \sum_{\ell \geq 1} 
 \,\frac{\ell^{r+1}\,A_{r}(\lambda^\ell)}{\lambda(1-\lambda^\ell)^{r+1}}.  
\]
Since $A_r(\lambda^\ell)/\lambda$ is a polynomial in $\lambda$ for each $r\ge 1$,  
 the above series is  easily seen to be bounded from above by a constant 
 $C_r$ if $\lambda<\frac 12$, say,  so we conclude from the 
mean value theorem that, if $t>\e^{-\frac \pi 2}$ 
(and then $\e^{-\frac{4\pi^2}{|\log t|}} <\e^{-8\pi}<\frac 12$),   
\begin{equation}
\label{karle}
\kappa_r(\e^{-4\pi^2/|\log t|}) \le C_r\,\e^{-4\pi^2/|\log t|}. 
\end{equation}
In addition recall that the Eulerian polynomials $A_r$  have non negative coefficients,
 and coefficient one  in the linear terms  for $r \ge 1$, to obtain from
    \eqref{f1kapa} by just taking the first term in each series  the following  lower bounds:  
\[
\kappa_r(\e^{-4\pi^2/|\log t|})
\ge\frac{ A_r(\e^{-4\pi^2/|\log t|})}{(1- \e^{-4\pi^2/|\log t|})^{r+1}}
\ge \e^{-4\pi^2/|\log t|}.
\]
Thus, \eqref{aEj}  holds. Indeed, from  \eqref{karle} and \eqref{fEj} we conclude
 that there is a positive  constant $C_j$ that may depend on $j$ such that for any $t>\e^{-\frac \pi 2}$,
 \begin{equation}
 \label{Ejle}
 |E_j(t)|\le \frac{C_j}{|\log t|^{2j}}\,\e^{-4\pi^2/|\log t|}. 
  \end{equation}
 \end{proof}
 We computed in Lemma \ref{lemacum} the derivatives at zero of
  the  cumulant generating function $K_{X(t)}$.
 To estimate the remainder in the Taylor formula we need also  the
  derivatives of  $K_{Z(t)}(\theta)$ 
  at $\theta\neq 0$.  
 Recall that 
\[
K_{Z(t)}(\theta)=K_{X(t)}\big(\theta/\sigma(t)\big)-i \theta\,\kappa_1(t)/\sigma(t).
\]
\begin{lemma}
\label{lemaderK}
For each $j\ge 2$,
$\theta\in \mathbb R$,  
\[ 
\partial^{(j)}_\theta  K_{Z(t)}(\theta) =
\frac{i^j}{(\sigma(t))^j}\,\kappa_j(t\, {\rm e}^{i\theta/\sigma(t)}),
\]
where the functions $\kappa_j(z)$ are defined for $z\in \mathbb C, |z|<1$ by formula \eqref{f1kapa}:  
\[
\kappa_j(z)=
  \sum_{\ell \geq 1} \,\frac{\ell^{j-1} \,A_{j}(z^\ell)}{(1-z^\ell)^{j+1}}.
  \]
\end{lemma}
\begin{proof} Observe first that, from  the same reasoning used to prove \eqref{f1kapa},
  the functions 
$\kappa_j(z)$ satisfy the recurrence $z\partial_z \kappa_j(z)=\kappa_{j+1}(z)$. 
From the expression $K_{Z(t)}(\theta)=L(t \e^{i\theta/\sigma(t)})-L(t)-i\theta
 \frac{\kappa_1(t)}{\sigma(t)}$
in terms of  $L(z):=\sum_{\ell\ge 1}\frac 1\ell \frac{z^\ell}{1-z^\ell}$,
  it is straightforward to see by differentiating the series that 
\begin{align}
\nn
\partial_\theta  K_{Z(t)}(\theta) =&\frac{i}{\sigma(t)}\Big(t \e^{i\theta/\sigma(t)}\,
 \sum_{\ell \geq 1} \,\frac{(t\, \e^{i\theta/\sigma(t)})^{\ell-1}}{(1-(t\, \e^{i\theta/\sigma(t)})^\ell)^{2}}
 -\kappa_1(t)\Big )
 \\ \nn
=&\frac{i}{\sigma(t)}\,\big(\kappa_1(t\, \e^{i\theta/\sigma(t)}) -\kappa_1(t)\big ).
\end{align}
Differentiating the last expression  we obtain, using the recurrence for the $\kappa_j$'s 
\[
\partial^{(2)}_\theta  K_{Z(t)}(\theta)=\big(\frac{i}{\sigma(t)}\big)^2\,t \,\e^{i\theta/\sigma(t)}
\kappa_1'(t \e^{i\theta/\sigma(t)})=\frac{i^2}{(\sigma(t))^2}\,\kappa_2(t\, \e^{i\theta/\sigma(t)}),
\]
which is the desired expression for $j=2$. Successive differentiation
 using the recurrence yields the general formula. 
\end{proof}
\medskip
\begin{corollary}
\label{derKZ}
The derivatives of $ K_{Z(t)}(\theta)$ satisfy
\begin{align}
\nn 
&{\rm{a)}}\quad|\partial^{(j)}_\theta  K_{Z(t)}(\theta) |\le \frac{\kappa_j(t)}{(\sigma(t))^j}
\\ \nn
&{\rm{b)}}\quad\frac{1}{j!}\,|\partial^{(j)}_\theta  K_{Z(t_n)}(\theta) |\le
C_j\, \big(\,\frac{2\pi^2}{3}\,(n-\frac{1}{24})+
\frac 14\,\big)^{\frac{2-j}{4}},\quad j\ge 3, 
 \end{align}
for some positive constant $C_j$ that may depend on $j$. 
\end{corollary}

\begin{proof}
From Lemma  \ref {lemaderK}, 
\[ |\partial^{(j)}_\theta  K_{Z(t)}(\theta)|=\frac{1}{(\sigma(t))^j} 
\big|\sum_{\ell \geq 1} \,
\frac{\ell^{j-1} \,A_{j}((t \e^{i\theta/\sigma(t)})^\ell)}{\big(1-(t \e^{i\theta/\sigma(t)})^\ell\big)^{j+1}}
\big|
\le
\frac{1}{(\sigma(t))^j} 
\sum_{\ell \geq 1} \,
\frac{\ell^{j-1} \,A_{j}(t^\ell)}{(1-t^\ell)^{j+1}}, 
\]
where to estimate  the denominators we use  that
$| 1-(t \e^{i\theta/\sigma(t)})^\ell|\ge 1-t^\ell$, 
and to estimate de numerators it is enough to observe that the Eulerian polynomials 
have real  positive coefficients. The right hand side above is precisely that in item a)
 in the statement. 
To prove the inequality b), we write, from \eqref{kj} 
for $\kappa_j(t)$ and
 $\sigma^2(t)=\kappa_2(t)$:
\[\frac{\kappa_j(t)}{j!\,(\sigma(t))^j}=\frac{1}{2}\,
\frac{\frac{\pi^2}{3|\log t|}-\frac 1j +\frac{2\,|\log t|^j}{j!}\,E_j(t)}
{\big(\frac{\pi^2}{3|\log t|}-\frac 12 +|\log t|^2\,E_2(t)\big)^{j/2}}.
\]
For each  $j\ge 3$, we take  $t=t_n$ above  and  use \eqref{Ejle} and  \eqref{tna}  to conclude that
{\rm b)} in the statement holds. 
\end{proof}

\end{document}